\documentclass[a4paper,11pt]{article}
\usepackage[latin1]{inputenc}
\usepackage[francais,english]{babel}    
\usepackage{amssymb}
\usepackage{amsmath}
\usepackage{amsthm}
\usepackage{multicol}
\usepackage{graphicx}
\usepackage{color}
\usepackage[T1]{fontenc}
\usepackage{yfonts}
\usepackage{placeins}
\usepackage{array}
\usepackage{dsfont}
\usepackage{stmaryrd}
\usepackage{verbatim}
\usepackage{caption}
\usepackage{transparent}

\newtheorem{thm}{Theorem}[section]
\newcommand{\bthm}{\begin{thm}}
\newcommand{\ethm}{\end{thm}}

\newtheorem{thmi}{Theorem}
\newcommand{\bthmi}{\begin{thmi}}
\newcommand{\ethmi}{\end{thmi}}

\newtheorem{cori}[thmi]{Corollary}
\newcommand{\bcori}{\begin{cori}}
\newcommand{\ecori}{\end{cori}}

\newtheorem{mthm}{Theorem}
\newcommand{\bmthm}{\begin{mthm}}
\newcommand{\emthm}{\end{mthm}}
\newtheorem{mcor}[mthm]{Corollary}
\newcommand{\bmcor}{\begin{mcor}}
\newcommand{\emcor}{\end{mcor}}

\newtheorem*{conj}{Conjecture}
\newcommand{\bconj}{\begin{conj}}
\newcommand{\econj}{\end{conj}}

\newtheorem*{question}{Question}
\newcommand{\bq}{\begin{question}}
\newcommand{\eq}{\end{question}}

\newtheorem*{thn}{Theorem}
\newcommand{\bthn}{\begin{thn}}
\newcommand{\ethn}{\end{thn}}

\newtheorem{exo}{Exercise}
\newcommand{\bex}{\begin{exo}}
\newcommand{\eex}{\end{exo}}

\newtheorem{sol}{Solution}
\newcommand{\bsol}{\begin{sol}}
\newcommand{\esol}{\end{sol}}

\newtheorem{pro}[thm]{Proposition}
\newcommand{\bpro}{\begin{pro}}
\newcommand{\epro}{\end{pro}}

\newtheorem{cor}[thm]{Corollary}
\newcommand{\bcor}{\begin{cor}}
\newcommand{\ecor}{\end{cor}}

\newtheorem{lem}[thm]{Lemma}
\newcommand{\blem}{\begin{lem}}
\newcommand{\elem}{\end{lem}}

\theoremstyle{definition}

\newtheorem{defi}[thm]{Definition}
\newcommand{\bdf}{\begin{defi}}
\newcommand{\edf}{\end{defi}}

\newtheorem*{defis}{Definition}
\newcommand{\bdfs}{\begin{defis}}
\newcommand{\edfs}{\end{defis}}

\newtheorem*{rmk}{Remark}
\newcommand{\brk}{\begin{rmk} \upshape}
\newcommand{\erk}{\end{rmk}}

\newtheorem*{rmks}{Remarks}
\newcommand{\brks}{\begin{rmks} \upshape}
\newcommand{\erks}{\end{rmks}}

\newtheorem*{exe}{Example}
\newcommand{\bexe}{\begin{exe} \upshape}
\newcommand{\eexe}{\end{exe}}

\newtheorem*{exes}{Examples}
\newcommand{\bexes}{\begin{exes} \upshape}
\newcommand{\eexes}{\end{exes}}

\newtheorem*{pre}{Proof}
\newcommand{\bp}{\begin{pre} \upshape}
\newcommand{\ep}{\hfill \qed \end{pre}}
\newcommand{\epp}{\end{pre}}

\newtheorem{mainthm}{Theorem}
\newtheorem{maincor}[mainthm]{Corollary}

\newcommand{\beq}{\begin{eqnarray*}}
\newcommand{\eeq}{\end{eqnarray*}}

\newcommand{\beqn}{\begin{equation}}
\newcommand{\eeqn}{\end{equation}}

\newcommand{\ben}{\begin{enumerate}}
\newcommand{\een}{\end{enumerate}}
\newcommand{\bit}{\begin{itemize} \renewcommand{\labelitemi}{$\bullet$} \renewcommand{\labelitemii}{$\star$}}
\newcommand{\eit}{\end{itemize}}

\newcommand{\bfg}{
\begin{figure}[H]
\begin{center}}
\newcommand{\efg}{
\end{center}
\end{figure}
\FloatBarrier}

\newcolumntype{M}[1]{>{\raggedright}m{#1}}

\newcommand{\R}{\mathbb{R}}

\newcommand{\N}{\mathbb{N}}
\newcommand{\Z}{\mathbb{Z}}

\newcommand{\E}{\mathbb{E}}

\newcommand{\K}{\mathbb{K}}

\newcommand{\F}{\mathbb{F}}

\newcommand{\bs}{\symbol{92}}

\newcommand{\ov}{\overline}

\newcommand{\Out}{\operatorname{Out}}

\newcommand{\st}{\, | \,}

\newcommand{\ra}{\rightarrow}

\newcommand{\liml}{\lim\limits}
\newcommand{\limlinf}{\liminf\limits}

\newcommand{\f}{\frac}

\renewcommand{\geq}{\geqslant}
\renewcommand{\leq}{\leqslant}

\newcommand{\SL}{\operatorname{SL}}

\renewcommand{\>}{\rangle}
\newcommand{\pif}{{+\infty}}

\newcommand{\mk}{\medskip}

\newcommand{\sign}{\begin{flushright}
Thomas Haettel \\
Universit\'e de Montpellier \\
Institut Montpelli\'erain Alexander Grothendieck \\
CC051 \\
Place Eug\`ene Bataillon \\
34095 Montpellier Cedex 5 \\
France \\
thomas.haettel@math.univ-montp2.fr
\end{flushright}}

\makeatletter
\def\Ddots{\mathinner{\mkern1mu\raise\p@
\vbox{\kern7\p@\hbox{.}}\mkern2mu
\raise4\p@\hbox{.}\mkern2mu\raise7\p@\hbox{.}\mkern1mu}}
\makeatother

\makeatletter
\def\maketitles{%
  \null
  \thispagestyle{empty}%
  \vfill
  \begin{center}\leavevmode
    \normalfont
    {\LARGE \@title\par}%
    \vskip 1.2cm
    {\large \@author\par}%
    \vskip 1.2cm
    {\large \@subtitle\par}%
    \vskip 0.8cm
    {\large \@date\par}%
  \end{center}%
  \vfill
  \null
  \cleardoublepage
  }

\def\date#1{\def\@date{#1}}
\def\author#1{\def\@author{#1}}
\def\title#1{\def\@title{#1}}
\def\subtitle#1{\def\@subtitle{#1}}
\makeatother

\newcommand{\calf}{\mathcal{F}}
\newcommand{\calt}{\mathcal{T}}

\newcommand{\calp}{\mathcal{P}}
\newcommand{\calz}{\mathcal{Z}}
\newcommand{\grp}[1]{\langle #1\rangle}
\newcommand{\normal}{\triangleleft}
\newcommand{\dunion}{\coprod}
\newcommand{\bbZ}{\mathbb{Z}}
\usepackage[all]{xy}
\usepackage{bbm}

\newtheorem{de}{Definition}
\newtheorem{theo}[de]{Theorem} 

\newtheorem{lemma}[de]{Lemma}
\newtheorem{coro}[de]{Corollary}

\theoremstyle{remark}

\title{Hyperbolic rigidity of higher rank lattices}
\author{Thomas Haettel \\ Appendix by Vincent Guirardel and Camille Horbez}
\date{\today}

\begin{document}

\selectlanguage{english}

\maketitle

\begin{center}
\begin{minipage}{0.8\textwidth}
\textsc{Abstract.} We prove that any action of a higher rank lattice on a Gromov-hyperbolic space is elementary. More precisely, it is either elliptic or parabolic. This is a large generalization of the fact that any action of a higher rank lattice on a tree has a fixed point. A consequence is that any quasi-action of a higher rank lattice on a tree is elliptic, i.e. it has Manning's property (QFA). Moreover, we obtain a new proof of the theorem of Farb-Kaimanovich-Masur that any morphism from a higher rank lattice to a mapping class group has finite image, without relying on the Margulis normal subgroup theorem nor on bounded cohomology. More generally, we prove that any morphism from a higher rank lattice to a hierarchically hyperbolic group has finite image. In the Appendix, Vincent Guirardel and Camille Horbez deduce rigidity results for morphisms from a higher rank lattice to various outer automorphism groups.
\end{minipage}
\end{center}

\let\thefootnote\relax\footnotetext{{\bf Keywords} : Higher rank lattices, Gromov hyperbolic, rigidity, property (T), coarse median spaces, quasi-actions, random walks, mapping class groups, hierarchically hyperbolic groups. {\bf AMS codes} : 22E40, 53C24, 20F67, 19J35, 05C81, 60J65}

\section*{Introduction}
\

Higher rank semisimple algebraic groups over local fields, and their lattices, are well-known to enjoy various rigidity properties. The main idea is that they cannot act on any other space than the ones naturally associated to the Lie group. This is reflected notably in the Margulis superrigidity theorem, and is also the motivating idea of Zimmer's program.

Concerning the rigidity of isometric actions, the most famous example is Kazhdan's property (T), which tells us that higher rank lattices cannot act by isometries without fixed point on Hilbert spaces. In fact, property (T) also implies such a fixed point property for some $L^p$ spaces (see~\cite{bader_furman_gelander_monod}), for trees (see~\cite{serre_fa}), and more generally for metric median spaces (such as CAT(0) cube complexes, see~\cite{chatterji_drutu_haglund}).

Property (T) is also satisfied notably by hyperbolic quaternionic lattices and by some random hyperbolic groups (see~\cite{zuk}). There have been various strengthenings of property (T), which are all satisfied by higher rank lattices but not by hyperbolic groups, which imply fixed point properties for various actions on various Banach spaces (see for instance~\cite{lafforgue_t} and \cite{monod}).

Since Gromov-hyperbolic spaces play a central role in geometric group theory, understanding actions of higher rank lattices on Gromov-hyperbolic spaces is an extremely natural question. There are several partial answers to that question, for instance any action on a tree or on a symmetric space of rank $1$ has a fixed point. Manning proved that, for $\SL(n,\Z)$ with $n \geq 3$ and some other boundedly generated groups, any action on quasi-tree is bounded (see~\cite{manning}). Using V.~Lafforgue's strengthened version of property (T) (see~\cite{lafforgue_t}, \cite{liao}, \cite{delaat_delasalle}), one deduces that if $\Gamma$ is a cocompact lattice in a higher rank semisimple algebraic group, then any action of $\Gamma$ by isometries on a uniformly locally finite Gromov-hyperbolic space is bounded.

The main purpose article is to prove the following.

\begin{mainthm} \label{thm:main}
Let $\Gamma$ be a lattice in (a product) of higher rank almost simple connected algebraic groups with finite centers over a local field. Then any action of $\Gamma$ by isometries on a Gromov-hyperbolic metric space is elementary. More precisely, it is either elliptic or parabolic.
\end{mainthm}

\brks\
\bit
\item This result has also been announced by Bader and Furman, as it should be a consequence of their deep work on rigidity and boundaries (see notably~\cite[Theorem~4.1]{bader_furman_icm} for convergence actions of lattices). However, the techniques are essentially different: Bader and Furman use a lot of ergodic theory, while in this article we use very little of it, and focus mostly on the asymptotic geometry of lattices and buildings, making a crucial use of medians.
\item One should note that the hyperbolic space in the theorem is not assumed to be locally compact, nor the action is assumed to satisfy any kind of properness. 
\item Note that most rigidity results conclude to the boundedness of orbits. Since any finitely generated group has a metrically proper, parabolic action on a hyperbolic space (locally infinite in general), one needs to include those parabolic actions (see for instance~\cite{hruska}).
\item Even though they do not appear in the statement, the theory of coarse median spaces developed by Bowditch (see~\cite{bowditch_coarse_median}) plays a crucial role in the proof.
\item In the theorem, we have to assume that each almost simple factor has higher rank. Our methods use the induction to the ambient group, so we cannot study irreducible lattices in products of rank $1$ groups. However, in this case, Bader and Furman can prove the following: for any irreducible lattice in a product of at least two groups, any isometric action on a hyperbolic space $X$ without bounded orbits in $X$ or finite orbits in $\partial X$, there is a closed subset of the boundary on which the action extends to one factor.
\item Whereas most rigidity results concerning higher rank lattices use bounded cohomology, Margulis superridigidity or normal subgroup theorems or V.~Lafforgue's strengthenings of property (T), our proof uses really new ingredients, and in particular medians spaces.
\eit
\erks

In~\cite{manning}, Manning was motivated by the question of quasi-actions of groups of trees. For the precise definition of a quasi-action, we refer to Section~\ref{sec:cor}. A straightforward consequence of Theorem~\ref{thm:main} is the following.

\begin{maincor} \label{cor:qfa}
Let $\Gamma$ be as in Theorem~\ref{thm:main}. Then any quasi-action of $\Gamma$ by isometries on a tree is elliptic. In other words, $\Gamma$ has Manning's property (QFA).
\end{maincor}

Another major consequence of Theorem~\ref{thm:main} is another proof of the following.

\begin{maincor}[Farb-Kaimanovich-Masur~\cite{farb_masur}, \cite{kaimanovich_masur}] \label{cor:mcg} 
Let $\Gamma$ be as in Theorem~\ref{thm:main}, and let $S_{g,p}$ be a closed surface of genus $g$ with $p$ punctures. Then any morphism $\Gamma \ra MCG(S_{g,p})$ has finite image.
\end{maincor}

The proof of Farb, Masur and Kaimanovich relies notably on the very deep Margulis normal subgroup theorem. Our purpose here is to give a proof as simple as possible, and we will not rely on any such deep theorem in the uniform case, and in the non-uniform one case we will use Margulis arithmeticity theorem only to ensure that the associated cocycle is integrable. In particular, in the proof of Corollary~\ref{cor:mcg}, we will not even use Burger-Monod's result that higher rank lattices do not have unbounded quasi-morphisms. We will simply use the fact that higher rank lattices do not surject onto $\Z$ (it is a direct consequence of property (T)) and use the weaker form of Theorem~\ref{thm:main} stating that every action of a higher rank lattice on a hyperbolic space is elementary.

\mk

In fact, we can apply the exact same proof as in Corollary~\ref{cor:mcg} to the more general class of hierarchically hyperbolic groups. They have been defined and studied in several articles (see~\cite{hhg1}, \cite{hhg2}, \cite{hhg3}, \cite{hhg4}), and since the definition is technical and irrelevant for the rest of the article, we refer to these articles for the precise definitions and main results. Roughly speaking, hierarchically hyperbolic spaces are metric spaces with a nice collection of projections to hyperbolic spaces, organized with some hierarchical structure. Notable examples of hierarchically hyperbolic groups include hyperbolic groups, mapping class groups, right-angled Artin groups, and they are stable under relative hyperbolicity.

\begin{maincor} \label{cor:hhg} 
Let $\Gamma$ be as in Theorem~\ref{thm:main}, and let $G$ be a hierarchically hyperbolic group. Then any morphism $\Gamma \ra G$ has finite image.
\end{maincor}

Another classical generalization of hyperbolic groups is the class of acylindrical hyperbolic groups, developed notably by Osin (see~\cite{osin}). A group is called acylindrically hyperbolic if it admits a non-elementary acylindrical action on a hyperbolic space. Classical examples include (relatively) hyperbolic groups, mapping class groups, outer automorphism groups of free groups, and many others. Following Mimura (see~\cite{mimura}), we say that a subgroup $H$ of an acylindrically hyperbolic group $G$ is absolutely elliptic if, for every acylindrical action of $G$ on a hyperbolic space, $H$ acts elliptically. Note that Mimura proved the following result for Chevalley groups.

\begin{maincor} \label{cor:ah} 
Let $\Gamma$ be as in Theorem~\ref{thm:main}, and let $G$ be an acylindrically hyperbolic group. Then any morphism $\Gamma \ra G$ has universally elliptic image.
\end{maincor}

Concerning morphisms from a higher rank lattice to $\Out(\F_n)$, Bridson and Wade proved that their image is also finite (see~\cite{bridson_wade}). However, there are subgroups of $\Out(F_n)$ with bounded orbits in the free splitting complex, but with no finite orbits, so we cannot give a proof in that case which is as simple as in the mapping class group case. Nevertheless, in the Appendix, Vincent Guirardel and Camille Horbez use Theorem~\ref{thm:main} to deduce several rigidity results for morphisms to various outer automorphism groups. Let us present the following result, and refer the reader to the Appendix for the other ones.

\begin{maincor} \label{cor:main_appendix} 
Let $\Gamma$ be as in Theorem~\ref{thm:main}, and let $G$ be a torsion-free hyperbolic group. Then any morphism $\Gamma \ra \Out(G)$ has finite image.
\end{maincor}

\mk

We will now give the outline of the proof of Theorem~\ref{thm:main}, and explain the different parts of the article.

In Section~\ref{sec:induction}, we show how to use $L^1$ induction to obtain, starting from an action of a higher rank lattice $\Gamma < G$ on a hyperbolic space, an action of $G$ on a coarse median space $Y$. To that purpose, if $\Gamma$ is non-uniform, we use Shalom's work on integrability of cocyles (see~\cite{shalom}).

In Section~\ref{sec:buildings}, we show that any action of $G$ on a coarse median space $Y$ has sublinear orbit growth. To prove this, we embed an asymptotic cone of $G$, which is a non-discrete affine building, into the asymptotic cone of $Y$, which is a metric median space. We then use techniques similar to~\cite{haettel_coarse_median}, where we proved that higher rank affine buildings do not admit any Lipschitz median.

In Section~\ref{sec:random walks}, we use the Brownian motion of the symmetric space of $G$, or a standard random walk on the $1$-skeleton of the Bruhat-Tits building of $G$, and use Lyons-Sullivan's discretization procedure to show that some random walk on $\Gamma$ has zero drift in $X$. Finally, we use a result of Maher and Tiozzo (Theorem~\ref{thm:maher tiozzo}) to show that the action of $\Gamma$ on $X$ is elementary. Finally, we rule out the case of lineal actions using Burger and Monod's result that $\Gamma$ has no unbounded quasimorphisms.

In Section~\ref{sec:cor}, we give the proof of the corollaries.

\mk

We would like to thank Yves Benoist for suggesting to look at Lyons-Sullivan's discretization procedure of the Brownian motion, and for long and stimulating discussions. We would like to thank Jean Lécureux, Bruno Duchesne, Nicolas Monod, Mikael de la Salle, Koji Fujiwara, Anthony Genevois, Brian Bowditch and Mark Hagen for interesting discussions.

\section{Definitions}

We start by recalling the definitions of medians and coarse medians as defined by Bowditch in~\cite{bowditch_coarse_median} and their essential properties.

\bdf A \emph{median} on a set $X$ is a map $\mu:X^3 \ra X$ which satisfies the following:
\bit \item[(M1)] $\forall a,b,c \in X, \mu(a,b,c)=\mu(b,a,c)=\mu(b,c,a)$,
\item[(M2)] $\forall a,b \in X, \mu(a,a,b)=a$,
\item[(M3)] $\forall a,b,c,d,e \in X, \mu(a,b,\mu(c,d,e))=\mu(\mu(a,b,c),\mu(a,b,d),e)$.
\eit
The pair $(X,\mu)$ is also called a \emph{median algebra}.
\edf

Furthermore, there exist universal objects called free median algebras, for which we will simply state the following.

\bpro \label{pro:free median} For any $p \in \N$, there exists a finite \emph{free median algebra} $(X,\mu_X)$ such that, for any finite median algebra $(Y,\mu_Y)$ with $|Y| \leq p$, there exists a median surjective homomorphism $(X,\mu_X) \ra (Y,\mu_Y)$. \epro

Medians are interesting from the viewpoint of geometry thanks to the following notion.

\bdf Let $(X,d)$ be a metric space. The \emph{interval} between $a,b \in X$ is $[a,b]=\{c \in X \st d(a,c)+d(c,b)=d(a,b)\}$. The metric space $(X,d)$ is called \emph{metric median} if for every $a,b,c \in X$, the intersection $[a,b] \cap [b,c] \cap [c,a]$ is a single point $\mu(a,b,c)$. \edf

Note that if $(X,d)$ is metric median, the fonction $\mu:X^3 \ra X$ given in the definition is a median.

\bexes\
\bit \item $\R$, $(\R^n,\ell_1)$ or any $L^1$ space are metric median spaces.
\item Products of metric median spaces, endowed with the $\ell_1$ product distance, are metric median.
\item $\{0,1\}$, and the $n$-cube $\{0,1\}^n$, are metric median, with the combinatorial distance.
\item According to~\cite{chepoi_median_cat0}, any simplicial graph is metric median if and only if it is the $1$-skeleton of a CAT(0) cube complex.
\item Any $\R$-tree is a metric median space.
\eit
\eexes

\bdf
Let $(X,\mu)$ be a median algebra. The \emph{rank} of $(X,\mu)$ is the supremum of integers $n \in \N$ such that there exists a median embedding of the $n$-cube $\{0,1\}^n \ra X$ into $X$. 
\edf

\bdf
Let $(X,\mu)$ be a median algebra. A subset $A$ of $X$ is called \emph{convex} if for every $a,b \in A$, we have $[a,b] \subset X$.
\edf

\bpro
Let $(X,\mu)$ be a median algebra. For any two distinct points $x,y \in X$, there exists a \emph{wall} $W=\{H^+,H^-\}$ separating $x$ and $y$, i.e. $X=H^+ \sqcup H^-$ is a partition of $X$ into two convex subsets $H^+$, $H^-$, such that $x$ and $y$ do not belong the same $H^\pm$.
\epro

In~\cite{bowditch_coarse_median}, Bowditch defined the notion of a coarse median space, in order to encompass notably hyperbolic spaces, CAT(0) cube complexes and mapping class groups. This is a natural generalization of the definition of Gromov-hyperbolic spaces using comparisons with finite metric trees. Roughly speaking, coarse median spaces have good uniform approximations by finite CAT(0) cube complexes.

\bdf
Let $(X,d)$ be a metric space. A \emph{coarse median} on $X$ is map $\mu : X^3 \ra X$ which satisfies $(M1)$, $(M2)$ and the following:
\bit \item[(C1)] There are constant $k$,$h(0)$ such that for all $a,b,c,a',b',c' \in X^3$, we have
$$ d(\mu(a,b,c),\mu(a',b',c')) \leq k(d(a,a')+d(b,b')+d(c,c'))+h(0).$$
\item[(C2)] There is a function $h:\N \ra [0,\infty)$ with the following property. Suppose that $A \subset X$ is finite with $|A| \leq p$, then there exists a finite median algebra $(\Pi,\mu_\Pi)$ and maps $\pi:A \ra \Pi$ and $\eta : \Pi \ra X$ such that
\beq  &\forall x,y,z \in \Pi, d(\eta \mu_\Pi(x,y,z),\mu(\eta x,\eta y,\eta z)) \leq h(p) \\
&\forall a \in A, d(a,\eta \pi a) \leq h(p).\eeq
\eit
If furthermore the median algebra $\Pi$ can always be chosen to have a rank bounded by $r$, we say that $\mu$ is a coarse median of rank at most $r$. 
\edf

\bexes\
\bit 
\item Any median metric space is coarse median.
\item Any metric space quasi-isometric to a coarse median space is coarse median.
\item A metric space is Gromov-hyperbolic if and only if it is coarse median of rank $1$.
\item Any space hyperbolic relative to coarse median spaces is coarse median (see~\cite{bowditch_rh_coarse_median}).
\item For any closed surface $S$ possibly with punctures, the mapping class group of $S$ and the Teichmüller space of $S$ with either the Teichmüller or Weil-Peterson metric are coarse median (see~\cite{bowditch_coarse_median}).
\item Any hierarchically hyperbolic space is coarse median (see~\cite{hhg2}).
\item Higher rank lattices are not coarse median (see~\cite{haettel_coarse_median}).
\eit
\eexes

One of the main tools to study coarse median spaces are asymptotic cones. 

\bthm[Bowditch, see~\cite{bowditch_coarse_median}] Let $(X,d,\mu)$ be a coarse median space. Then on any asymptotic cone $(X_\infty,d_\infty)$, there is a canonically defined median $\mu_\infty:X_\infty^3 \ra X_\infty$, which is $d_\infty$-Lipschitz with respect to each variable. \ethm

\section{Induction of the action to the semisimple group}

\label{sec:induction}

For this section, we will use the following notations and assumptions.

\mk

$G$ is a locally compact group, compactly generated, and $\Gamma$ is a lattice in $G$. Fix a geodesic Gromov-hyperbolic space $(X,d_X)$, and consider an action of $\Gamma$ by isometries on $X$. Fix a basepoint $x_0 \in X$.

\mk

In this section, we will see how to produce an action of $G$ by isometries on a coarse median space using induction.

\mk

Let $\mu_X : X^3\ra X$ denote a coarse median on $X$. It is uniquely defined up to a distance bounded above by $\delta$, where $\delta \geq 0$ is a constant in the thin triangle definition of the Gromov-hyperbolicity of $X$. As a consequence, the action of $\Gamma$ on $X$ quasi-preserves $\mu_X$.

Since $\Gamma$ is a lattice in $G$, we can consider a measurable closed fundamental domain $U \subset G$ that contains a neighbourhood of $e$, such that $G=U\Gamma$. Let $\lambda$ denote the Haar probability measure on $G / \Gamma \simeq U$.

Let $Y=L^1(G/\Gamma,X)=\{a:U \ra X \mbox{ measurable } \st \int_U d_X(a(u),x_0) d\lambda(u) < +\infty\}$. Endow $Y$ with the $L^1$ distance, for $a,b \in Y$:
$$d_Y(a,b) = \int_U d_X(a(u),b(u)) d\lambda(u).$$
Define $\mu_Y:Y^3 \ra Y$ by $\mu_Y(a,b,c) : u \in U \mapsto \mu_X(a(u),b(u),c(u))$. 

\bpro
The space $(Y,d_Y,\mu_Y)$ is a coarse median space.
\epro

\bp
Let $k  \geq 0$ and $h:\N \ra [0,\pif)$ denote the constants in the definition of the coarse median $\mu_X$ on $X$. For any $a,b,c,a',b',c' \in Y$, we have
\beq d_Y(\mu_Y(a,b,c),\mu_Y(a',b',c')) = \int_U d_X(\mu_X(a(u),b(u),c(u)),\mu_X(a'(u),b'(u),c'(u))) d\lambda(u)& \\ \leq \int_U \left(k(d_X(a(u),a'(u))+d_X(b(u),b'(u))+d_X(c(u),c'(u)))+h(0)\right)d\lambda(u)&\\
\leq k(d_Y(a,a')+d_Y(b,b')+d_Y(c,c'))+h(0),&\eeq
so $\mu_Y$ satisfies the condition (C1).

Let $A \subset Y$ be a finite subset with $|A| \leq p$. For each $u \in U$, consider the finite subset $A(u) \subset X$: there exists a finite median algebra $(\Pi(u),\mu_{\Pi(u)})$ and maps $\pi(u):A(u) \ra \Pi(u)$, $\eta(u) : \Pi(u) \ra X$, such that for every $u \in U$, we have
\beq  &\forall x,y,z \in \Pi(u), d_X(\eta(u) \mu_{\Pi(u)}(x,y,z),\mu_X(\eta(u) x,\eta(u) y,\eta(u) z)) \leq h(p) \\
&\forall a \in A(u), d_X(a,\eta(u) \pi(u) a) \leq h(p).\eeq

Without loss of generality, one cas assume that, for every $u \in U$, the median algebra $(\Pi(u),\mu_{\Pi(u)})$ is a free median algebra over $p$ generators from Proposition~\ref{pro:free median}, which we denote simply $(\Pi,\mu_\Pi)$. Furthermore, we can assume that each of the maps $\pi(u):A \ra \Pi$ is constant, equal to some $\pi_Y:A \ra \Pi$.

Up to a uniformly bounded error $K \geq 0$, one may assume that, for each $x \in \Pi$, the map $u \in U \mapsto \eta(u)x \in X$ is measurable. Let us define the map $\eta_Y : \Pi \ra Y$ which to $x \in \Pi$ maps $\eta_Y(x) \in Y$ defined by $\eta_Y(x)(u) = \eta(u)(x)$.

For every $x,y,z \in \Pi$, we then have 
\beq &d_Y(\eta_Y \mu_{\Pi}(x,y,z),\mu_Y(\eta_Y(x),\eta_Y(y),\eta_Y(z))) = \\
&\int_U d_X\left(\eta(u) \mu_\Pi(x,y,z),\mu_X(\eta(u)(x),\eta(u)(y),\eta(u)(z))\right) d\lambda(u) \leq h(p).\eeq

Furthermore, for every $a \in A$, we have
$$ d_Y(a,\eta_Y \pi_Y a) = \int_U d_X\left(a(u),\eta(u)(\pi_Y(a))\right) d\lambda(u) \leq h(p),$$
so $\mu_Y$ satisfies the condition (C2).

As $\mu_Y$ also satisfies the conditions (M1) and (M2), this proves that $(Y,d_Y,\mu_Y)$ is a coarse median space (of infinite rank in general).\ep

\bpro
Any asymptotic cone of $(Y,d_Y,\mu_Y)$ is a metric median space.
\epro

\bp

Fix $\omega$ a non-principal ultrafilter on $\N$, fix $(y_n)_{n \in \N}$ a sequence of basepoints in $Y$, and fix a sequence $(\lambda_n)_{n \in \N}$ of scaling parameters going to $+\infty$. Consider the asymptotic cone $(Y_\infty,d_{Y,\infty},y_\infty,\mu_{Y,\infty})=\liml_\omega (Y,\f{1}{\lambda_n}d_Y,y_n,\mu_Y)$. Then according to~\cite{bowditch_coarse_median}, $\mu_{Y,\infty}$ is a Lipschitz median on $(Y_\infty,d_{Y,\infty})$. We will show that the metric $d_{Y,\infty}$ is actually a median metric, and the associated median is $\mu_{Y,\infty}$.

\mk

We will first show that $\mu_{Y,\infty}$-intervals are included in $d_{Y,\infty}$-intervals in $Y_\infty$. More precisely, fix $a_\infty=(a_n)_{n \in \N},b_\infty=(b_n)_{n \in \N}, c_\infty=(c_n)_{n \in \N}$ in $Y_\infty$ such that $\mu_{Y,\infty}(a_\infty,b_\infty,c_\infty)=b_\infty$. We will show that $d_{Y,\infty}(a_\infty,b_\infty)+d_{Y,\infty}(b_\infty,c_\infty)=d_{Y,\infty}(a_\infty,c_\infty)$. For each $n \in \N$, let $m_n=\mu_Y(a_n,b_n,c_n) \in Y$. By assumption, we have $\liml_{\omega} \f{d_Y(m_n,b_n)}{\lambda_n}=0$.

Since $m_n=\mu_Y(a_n,b_n,c_n)$, we know that for almost every $u \in U$, we have $m_n(u)=\mu_X(a_n(u),b_n(u),c_n(u))$. Since $X$ is Gromov-hyperbolic with constant $\delta \geq 0$, we know that $d_X(a_n(u),c_n(u)) \geq d_X(a_n(u),b_n(u))+d_X(b_n(u),c_n(u))-\delta$. By integrating over $U$, we obtain $d_Y(a_n,c_n) \geq d_Y(a_n,b_n)+d_Y(b_n,c_n)-\delta$. Passing to the ultralimit, we have $d_{Y,\infty}(a_\infty,c_\infty) \geq d_{Y,\infty}(a_\infty,b_\infty)+d_{Y,\infty}(b_\infty,c_\infty)$, since the sequences $(m_n)_{n \in \N}$ and $(b_n)_{n \in \N}$ define the same point in $Y_\infty$.

\mk

Conversely, we will show that $d_{Y,\infty}$-intervals are included in $\mu_{Y,\infty}$-intervals in $Y_\infty$. More precisely, fix $a_\infty=(a_n)_{n \in \N},b_\infty=(b_n)_{n \in \N}, c_\infty=(c_n)_{n \in \N}$ in $Y_\infty$ such that $d_{Y,\infty}(a_\infty,b_\infty)+d_{Y,\infty}(b_\infty,c_\infty)=d_{Y,\infty}(a_\infty,c_\infty)$. We will show that $\mu_{Y,\infty}(a_\infty,b_\infty,c_\infty)=b_\infty$. 

For each $n \in \N$, let $m_n=\mu_Y(a_n,b_n,c_n) \in Y$. So for almost every $u \in U$, we have $m_n(u)=\mu_X(a_n(u),b_n(u),c_n(u))$. Since $X$ is Gromov-hyperbolic, there exists a constant $\delta'\geq 0$ such that $d_X(m_n(u),b_n(u)) \leq d_X(a_n(u),b_n(u))+d_X(b_n(u),c_n(u))-d_X(a_n(u),c_n(u))+\delta'$.
By integrating over $U$, we obtain $d_Y(m_n,b_n) \leq d_Y(a_n,b_n)+d_Y(b_n,c_n)-d_Y(a_n,c_n)+\delta'$. Passing to the ultralimit, we have $d_{Y,\infty}(m_\infty,b_\infty) \leq d_{Y,\infty}(a_\infty,b_\infty)+d_{Y,\infty}(b_\infty,c_\infty)-d_{Y,\infty}(a_\infty,c_\infty)=0$. As a consequence, we obtain that $\mu_{Y,\infty}(a_\infty,b_\infty,c_\infty)=m_\infty=b_\infty$.

\mk

As a consequence, the asymptotic cone $(Y_\infty,d_{Y,\infty},y_\infty,\mu_{Y,\infty})$ is a metric median space.
\ep

Let us denote the projection map $P: G \ra G / \Gamma \simeq U$, and $\chi : G \ra \Gamma$ the map such that $\forall g \in G, g=\pi(g) \chi(g)$. This enables us to define the following map :
\beq \pi:G \times Y & \ra & Y \\
(g,a) & \mapsto & \left(g \cdot a : u \in U \mapsto \chi(g^{-1}u)^{-1} \cdot a(P(g^{-1}u))\right).\eeq
It is simply the natural $G$-action by left multiplication on the induced representation on  $Y=L^1(G/\Gamma,X)$. To see that the map $\pi$ is well-defined, we need the following integrability condition, where $d_\Gamma$ denote the word length of $\Gamma$ with respect to some finite generating set $S$:
\begin{equation} \forall g \in G, \int_U d_\Gamma(\chi(g^{-1}u),e) d\lambda(u) < \infty. \label{eq:integrability}\end{equation}

\mk

It should be noted that when $\Gamma$ is a uniform lattice in $G$, the integrability condition~(\ref{eq:integrability}) is satisfied. When $\Gamma$ is non-uniform and is as in Theorem~\ref{thm:main}, according to the Margulis arithmeticity theorem, $\Gamma$ is an arithmetic lattice. Hence, according to Shalom (see~\cite{shalom}), for every $g \in G$, the cocycle $\chi:u \in U \mapsto \chi(g^{-1}u) \in \Gamma$ is in $L^2(U,\lambda)$, so the integrability condition~(\ref{eq:integrability}) is satisfied.

\bpro
The map $\pi:G \times Y \ra Y$ is an action of $G$ on $Y$, by isometries, quasi-preserving $\mu_Y$.
\epro

\bp We will first show that $\pi$ is well-defined, using the integrability condition~(\ref{eq:integrability}). Let $M=\max_{\gamma \in S} d_X(\gamma \cdot x_0,x_0) \geq 0$: we have $\forall \gamma \in \gamma, d_X(\gamma \cdot x_0,x_0) \leq Md_\Gamma(\gamma,e)$. As a consequence, for every $g \in G$ and $a \in Y$, we have $\forall u \in U, d_X(\chi(g^{-1}u)^{-1} \cdot a(P(g^{-1}u)), x_0) \leq d_X(a(P(g^{-1}u)), x_0) + Md_\Gamma(\chi(g^{-1}u),e)$ so
\beq &&\int_U d_X(\chi(g^{-1}u)^{-1} \cdot a(P(g^{-1}u)), x_0) d\lambda(u) \\
&&\leq \int_U d_X(a(P(g^{-1}u)), x_0) d\lambda(u) + M\int_U d_\Gamma(\chi(g^{-1}u),e) d\lambda(u)< \infty,\eeq
since $a \in Y$ and by the integrability condition~(\ref{eq:integrability}). As a consequence, $\pi$ is well-defined.

\mk

We will now show that $\pi$ is an action. Let $g,h \in G$, $a \in Y$ and $u \in U$. Then
\beq \pi(g,\pi(h,a))(u) &=& \chi(g^{-1}u)^{-1} \cdot \pi(h,a)(P(g^{-1}u)) \\
&=& \chi(g^{-1}u)^{-1} \chi(h^{-1}P(g^{-1}u))^{-1} \cdot a(P(h^{-1}P(g^{-1}u))).\eeq
Notice that $\chi(h^{-1}P(g^{-1}u))\chi(g^{-1}u)=\chi(h^{-1}g^{-1}u)$ and $P(h^{-1}P(g^{-1}u))=P(h^{-1}g^{-1}u)$, so that 
$$\pi(g,\pi(h,a))(u) = \chi((gh)^{-1}u)^{-1} \cdot a(P((gh)^{-1}u)) = \pi(gh,a)(u).$$
As a consequence, $\pi$ is an action, and we will simply denote it $\pi(g,a)=g \cdot a$ for simplicity.

\mk

We will now show that $\pi$ is an action by isometries: let $g \in G$ and $a,b \in Y$. Then
\beq d_Y(g \cdot a,g \cdot b) &=& \int_U d_X(\chi(g^{-1}u)^{-1} \cdot a(P(g^{-1}u)),\chi(g^{-1}u)^{-1} \cdot b(P(g^{-1}u))) d\lambda(u) \\
&=& \int_U d_X(a(P(g^{-1}u)),b(P(g^{-1}u))) d\lambda(u)\\
&=& \int_U d_X(a(v),b(v)) d\lambda(v)\\
&=& d_Y(a,b),\eeq
since $u \mapsto P(g^{-1}u)$ is a measurable bijection from $U$ to $U$ which preserves the Haar measure $\lambda$.

\mk

We will show that this action quasi-preserves the coarse median $\mu_Y$. Let $C\geq 0$ such that $\forall \gamma \in \Gamma, \forall x,y,z \in X, d_X(\mu_X(\gamma \cdot x,\gamma \cdot y,\gamma \cdot z),\gamma \cdot \mu_X(x,y,z)) \leq C$. Then for any $g \in G$ and any $a,b,c \in Y$, we have
\beq &&d_Y(\mu_Y(g \cdot a,g \cdot b,g \cdot c),g \cdot \mu_Y(a,b,c)) =\\
&&\int_U d_X\left[\mu_X(\chi(g^{-1}u)^{-1} \cdot a(P(g^{-1}u)),\chi(g^{-1}u)^{-1} \cdot b(P(g^{-1}u)),\chi(g^{-1}u)^{-1} \cdot c(P(g^{-1}u))),\right.\\
&&\left.\chi(g^{-1}u)^{-1} \cdot \mu_X(a(P(g^{-1}u)),b(P(g^{-1}u)),c(P(g^{-1}u)))\right] d\lambda(u) \leq C.\eeq
\ep

Let $d_G$ denote any word quasi-metric on $G$ defined by a compact neighbourhood of the identity $B$ in $G$ which spans $G$, or any metric quasi-isometric to it.

\blem \label{lem:orbit coarse lipschitz}
The orbit map $g \in G \mapsto g \cdot y_0 \in Y$ is coarsely Lipschitz (with respect to $d_G$ and $d_Y$), i.e. there exist constants $K,C \geq 0$ such that
$$ \forall g,h \in G, d_Y(g \cdot y_0,h \cdot y_0) \leq Kd_G(g,h)+C.$$ 
\elem

\bp
Since the statement is independent of the quasi-isometry class of $d_G$, we will consider the word quasi-metric defined by $B$.

\bit
\item If $\Gamma$ is a uniform lattice, the fundamental domain $U$ can be chosen to be relatively compact, so $B^{-1}U$ is relatively compact in $G$. As a consequence, there exists a finite set $S \subset \Gamma$ such that $B^{-1}U \subset US$. Let $\alpha = \max_{\gamma \in S} d_X(\gamma \cdot x_0,x_0) \geq 0$. 
\item If $\Gamma$ is a non-uniform lattice, according to~\cite{shalom}, since $B$ is relatively compact, there exists $\beta>0$ such that
$$ \forall g \in B, \int_U d_\Gamma(\chi(g^{-1}u),e) d\lambda(u) \leq \beta.$$
There exists a constant $C>0$ such that $\forall \gamma \in \Gamma, d_X(\gamma \cdot x_0,x_0) \leq Cd_\Gamma(\gamma,e)$. As a consequence, we have 
$$ \forall g \in B, d_Y(g \cdot y_0,y_0) \leq C \int_U d_\Gamma(\chi(g^{-1}u),e) d\lambda(u) \leq C \beta.$$
Let $\alpha = C\beta$.
\eit

In either case, we have $\forall g \in B, d_Y(g \cdot y_0,y_0) \leq \alpha$. We can conclude that the map $g \in G \mapsto g \cdot y_0 \in Y$ is $\alpha$-Lipschitz.
\ep

\section{Actions of higher rank semisimple groups on coarse median spaces}

\label{sec:buildings}

For this section, we will use the following notations and assumptions.

$G$ is any finite product of higher rank almost simple connected algebraic groups with finite center over local field.

Let $K$ be a maximal compact subgroup of $G$, and consider left $G$-invariant, right $K$-invariant distance $d_G$ on $G$.

Fix a coarse median space $(X,d,\mu)$, and assume that $G$ acts by isometries on $X$, quasi-preserving $\mu$. Fix a basepoint $x_0 \in X$, and assume that the orbit map $g \in G \mapsto g \cdot x_0$ is coarsely Lipschitz. Assume furthermore that $(X,d,\mu)$ has metric median asymptotic cones.

The purpose of this section is to prove the following theorem. 

\bthm \label{thm:sublinear} $G$ has sublinear orbit growth, i.e.
$$ \liml_{R \ra \pif} \sup_{g \in G, d_G(e,g) \leq R} \f{d_X(g \cdot x_0,x_0)}{R} = 0.$$
\ethm

One can restrict to the case where $G$ is almost simple non-compact to prove Theorem~\ref{thm:sublinear}: we will now restrict to that case. Then $G$ is an almost simple non-compact algebraic group over a local field $\K$ of higher rank.

\subsection{Non-existence of loxodromic elements}

Recall that an isometry of a metric space $X$ is called loxodromic if there exists (equivalently, for every) $x \in X$ such that $\limlinf_{n \ra \pif} \f{d(g^n \cdot x,x)}{n} >0$.

Will show that Theorem~\ref{thm:sublinear} is a consequence of the following.

\bthm \label{thm:no loxodromic} No $\K$-semisimple element of $G$ acts loxodromically on $X$.
\ethm

\bp[Theorem~\ref{thm:no loxodromic} implies Theorem~\ref{thm:sublinear}]
By contraposition, let us assume that $G$ has linear orbit growth, so there exists an unbounded sequence $(g_n)_{n \in \N}$ in $G$ such that $\liml_{n \ra \pif} \f{d_X(g_n \cdot x_0,x_0)}{d_G(e,g_n)}=L>0$.

Consider a Cartan decomposition $G=KAK$, where $A$ is a maximal $\K$-split torus of $G$. Fix $a_1,\dots,a_r \in A$ that span a cocompact $\Z^r$ subgroup of $A$.

Since $K$ is compact and the action is coarsely Lipschitz, we may assume that $\forall n \in \N, g_n \in A$. There exist integers $d_{1,n},\dots,d_{r,n} \in \N$ such that $d_G(g_n,a_1^{d_{1,n}}\dots a_r^{d_{r,n}})$ is bounded with respect to $n \in \N$.

We know that there exists $1 \leq i \leq r$ such that $\liml_{n \ra \pif} \f{d_X(a_i^{d_{i,n}} \cdot x_0,x_0)}{d_G(e,g_n)}>0$. Since $d_G(g_n,e)$ is coarsely equivalent to $\max(d_G(a_i^{d_{i,n}},e), 1 \leq i \leq r)$, we have $\liml_{n \ra \pif} \f{d_X(a_i^{d_{i,n}} \cdot x_0,x_0)}{d_G(e,a_i^{d_{i,n}})}>0$. This proves that the $\K$-semisimple element $a_i$ acts loxodromically on $X$.
\ep

\subsection{Reduction to rank $2$}

We show that we can reduce our study to the case where $G$ is almost simple with $\K$-rank $2$.

\bpro \label{pro:rank 2} Assume that some $\K$-semisimple element of $G$ acts loxodromically on $X$. Consider a simple subgroup $H$ of $G$ defined over $\K$ of $\K$-rank $2$. Then some $\K$-semisimple element of $H$ acts loxodromically on $X$. \epro

\bp Let $A$ be a maximal $\K$-split torus of $G$, which contains a maximal $\K$-split torus $A'$ of $H$. Up to conjugation, we may assume some element $g_0 \in A$ acts loxodromically on $X$. Since $G$ is simple, there exists a finite number of elements $w_1,\dots,w_n$ in the (spherical) Weyl group of $G$ such that $A=\prod_{i=1}^n w_iA'w_i^{-1}$. Consider $h_1,\dots,h_n \in A'$ such that $g_0=\prod_{i=1}^n w_ih_iw_i^{-1}$. Since for every $1 \leq i \leq n$, the elements $w_ih_iw_i^{-1}$, for $1 \leq i \leq n$ pairwise commute, we know that for at least one $1 \leq i \leq n$, the element $h_i \in H$ acts loxodromically on $X$. 
\ep

In conclusion, in order to prove Theorem~\ref{thm:no loxodromic} we will reduce to the case where $G$ is almost simple with $\K$-rank $2$.

\subsection{Passing to asymptotic cones}

We show how to define a natural map from an asymptotic cone of $G$ to an asymptotic cone of the coarse median space $X$.

Fix a non-principal ultrafilter $\omega$ on $\N$. Define $(X_\infty,x_\infty,d_\infty,\mu_\infty)$ to be the $\omega$-ultralimit of $(X,x_0,\frac{1}{n}d,\mu)$: according to\cite{bowditch_coarse_median}, $\mu_\infty$ is a Lipschitz median on $(X_\infty,d_\infty)$. By assumption, $(X_\infty,d_\infty,\mu_\infty)$ is a metric median space.

Define $(G_\infty,e_\infty,d_{G_\infty})$ to be the $\omega$-ultralimit of $(G,e,\frac{1}{n}d_G)$. According to \cite{kleiner_leeb}, $(G_\infty,e_\infty,d_{G_\infty})$ is a non-discrete affine building.

\blem \label{lem:orbit map} The map $\phi:G_\infty \ra X_\infty : [g_n] \mapsto [g_n \cdot x_0]$ is well-defined and Lipschitz. \elem

\bp If $[g_n]=[g'_n]$, then by definition $\liml_\omega \frac{d_G(g_n,g'_n)}{n} =0$. Let $K,C \geq 0$ denote the constants for the definition of the orbit map $g \in G \mapsto g \cdot x_0$ being coarsely Lipschitz. Since $d(g_n \cdot x_0,g'_n \cdot x_0) \leq Kd_G(g_n,g'_n)+C$, we deduce that $\liml_\omega \frac{d(g_n \cdot x_0,g'_n \cdot x_0)}{n}=0$, hence $\phi$ is well-defined.

Furthermore, if $[g_n],[g'_n] \in G_\infty$, then since $d(g_n \cdot x_0,g'_n \cdot x_0) \leq Kd_G(g_n,g'_n)+C$, we deduce that $d_\infty([g_n \cdot x_0],[g'_n \cdot x_0]) \leq K d_{G_\infty}([g_n],[g'_n])$, so $\phi$ is $K$-Lipschitz.\ep

In view of a proof of Theorem~\ref{thm:no loxodromic} by contradiction, we will show a consequence of the existence of some loxodromic element.

\blem \label{lem:geodesic image of flats} Assume that some $\K$-semisimple element $g_0 \in G$ acts loxodromically on $X$. Then for any $\K$-semisimple element $g \in G$, the image of $([g^{\lfloor tn \rfloor}])_{t \in \R}$ under $\phi$ is a constant speed geodesic in $X_\infty$. \elem

\bp
Up to conjugation, we may assume that $g_0$ and $g$ belong to the same $\K$-split torus $A$. Since $G$ does not have relative type $A_1^2$, there exist conjugates (by Weyl group elements) $g_1$, $g_2$ of $g_0$ inside $A$ such that $\<g_0,g_1\>$, $\<g_0,g_2\>$ and $\<g_1,g_2\>$ are all cocompact ($\Z^2$) subgroups of $A$.

There exist $x,y \in \R$ such that $g=g_0^xg_1^y$. Up to using possibly $g_2$ instead of $g_0$ or $g_1$, we may assume that $|x| \neq |y|$, for instance $|x| > |y|$.

Let $L_0 = \liml_{n \ra \pif} \f{d(g_0^n \cdot x_0,x_0)}{n}>0$ by assumption. Since $g_1$ is a conjugate of $g_0$, we also have $\liml_{n \ra \pif} \f{d(g_1^n \cdot x_0,x_0)}{n}=L_0$. Then
\beq \liml_{n \ra \pif} \f{d(g^n \cdot x_0,x_0)}{n} & \geq & \f{d(g_0^{xn} \cdot x_0,x_0)}{n}-\f{d(g_1^{yn} \cdot x_0,x_0)}{n}\\
& \geq & L_0|x|-L_0|y|.\eeq
Let $L=\liml_{n \ra \pif} \f{d(g^n \cdot x_0,x_0)}{n} \geq L_0(|x|-|y|)>0$. Fix $t,s \in \R$. Then  \beq d_\infty([g^{\lfloor tn \rfloor}],[g^{\lfloor sn \rfloor}]) &=& \liml_\omega \f{d(g^{\lfloor tn \rfloor} \cdot x_0,g^{\lfloor sn \rfloor} \cdot x_0)}{n} \\
&=& \liml_\omega \f{d(g^{\lfloor (t-s)n \rfloor} \cdot x_0,x_0)}{n} = L|t-s|.\eeq
\ep

\subsection{Embedding of buildings into median spaces}

We will now prove a rigidity result for Lipschitz embeddings of affine buildings into metric median spaces. It could be stated in a more general form, but for simplicity we only state it in the way we will use it in the proof of Theorem~\ref{thm:no loxodromic}.

\bpro \label{pro:no embedding} There exists a $\K$-semisimple element $g \in G$ such that the image of $([g^{\lfloor tn \rfloor}])_{t \in \R}$ under $\phi$ is not a constant speed geodesic in $X_\infty$. \epro

\bp By contradiction, assume that for every $\K$-semisimple element $g \in G$, the image of $([g^{\lfloor tn \rfloor}])_{t \in \R}$ under $\phi$ is a constant speed geodesic in $X_\infty$.

Consider a maximal $\K$-split torus $A$ of $G$, and consider some affine line $L$ in the asymptotic cone $A_\infty \simeq \R^2$: we will show that $\phi(L)$ is a constant speed geodesic in $X_\infty$. The line $L$ could be parametrized as $([h^{\lfloor sn \rfloor}g^{\lfloor tn \rfloor}])_{t \in \R}$, where $g,h \in A$, $g \neq e$ and $s \in \R$. Informally speaking, it is the line through the point $h^{s\infty} \in A_\infty$, with direction given by $g^\infty \in A_\infty$. By assumption, the image of $([g^{\lfloor tn \rfloor}])_{t \in \R}$ under $\phi$ is a constant speed geodesic in $X_\infty$, so it is also the case for its translate, the image of $L$ under $\phi$.

\mk

Fix an apartment $F$ of $G_\infty$ which is the asymptotic cone of some maximal $\K$-split torus of $G$. Consider a wall $\{H^+,H^-\}$ in $X_\infty$ which separates some points in $\phi(F)$. According to the previous part, affine segments in $\phi(F) \simeq \R^2$ are geodesic. Since $H^\pm$ is convex in $X_\infty$, we deduce that $H^\pm \cap \phi(F)$ is affinely convex in $\phi(F)$.

The partition of $\phi(F) \simeq \R^2$ into two non-empty affinely convex subsets $H^\pm \cap \phi(F)$ determines an affine line $\phi(L)=\ov{H^+ \cap \phi(F)} \cap \ov{H^- \cap \phi(F)} \subset \phi(F)$ such that each connected component of $\phi(F \bs L)$ is contained in $H^+ \cap \phi(F)$ or in $H^- \cap \phi(F)$.

Since $G_\infty$ has spherical type different from $A_1^2$, there exists a singular line $L'$ in $F$ containing the basepoint $(1)_{n \in \N} \in F$ which intersects $L$ and is not orthogonal to $L$ (and $L' \neq L$). As a consequence, there exist two apartments $F_1,F_2$ in $G_\infty$, which are asymptotic cones of maximal $\K$-split tori of $G$, such that $F_1 \cap F$ and $F_2 \cap F$ are the two half-apartments of determined by $L'$, and furthermore such that $F_1 \cap F_2$ is a half-apartment $E$ bounded by $L'$. See Figure~\ref{fig:three}.

\begin{figure}[!h]
\def\svgwidth{8cm}
\center
\input{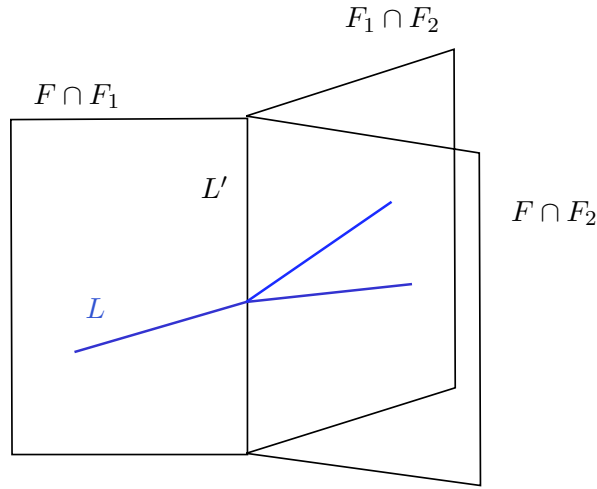}
\caption{The three half-flats}
\label{fig:three}
\end{figure}

Fix $i \in \{1,2\}$. Since the wall $\{H^+,H^-\}$ separates some points in $\phi(F_i)$, we can consider the affine line $\phi(L_i)=\ov{H^+ \cap \phi(F_i)} \cap \ov{H^- \cap \phi(F_i)} \subset F_i$. By uniqueness of $L$ and $L_i$, we have $L_i \cap F = L \cap F_i$. So $L_i$ is the affine line in $F_i$ containing the half-line $L \cap F_i$.

Since the wall $\{H^+,H^-\}$ separates some points in the half-plane $\phi(E)$, then $\ov{H^+ \cap \phi(E)} \cap \ov{H^- \cap \phi(E)}$ is a unique half-line, so $E \cap L_1 = E \cap L_2$. This contradicts the fact that $L$ and $L'$ are not orthogonal. This concludes the proof by contradiction.
\ep

We can now conclude the proof of Theorem~\ref{thm:no loxodromic}. According to Lemma~\ref{lem:geodesic image of flats} and Proposition~\ref{pro:no embedding} we deduce that no $\K$-semisimple element of $G$ acts loxodromically on $X$. According to Proposition~\ref{pro:rank 2}, this finishes the proof of Theorem~\ref{thm:no loxodromic} in the general case.

\section{Random walks on lattices}

\label{sec:random walks}

\subsection{Reduction to a random walk on the lattice}

We will now use the same notations as in Section~\ref{sec:induction}. We will use Theorem~\ref{thm:sublinear} to deduce the following.

\bpro \label{pro:zero drift}
Under the assumptions of Theorem~\ref{thm:main}, there exists a probability measure $\nu$ on $\Gamma$ (with infinite support, generating $\Gamma$), such that the associated random walk $(\gamma_n \cdot x_0)_{n \in \N}$ on $X$ has zero drift:
$$\liml_{n \ra \pif} \f{\E\left[d_X(\gamma_n \cdot x_0,x_0)\right]}{n} = 0.$$
\epro

\bp 

Consider a random variable $h$ on the fundamental domain $U \subset G$, following the Haar probability mesaure $\lambda$. The first objective is to build a family of random variables $(g_t)_{t \geq 0}$ in $G$, independent from $h$, a sequence of random stopping times $(N_k)_{k \geq 1}$, and a symmetric random walk $(\gamma_k)_{k \geq 1}$ on $\Gamma$, such that :
\bit
\item There exists a constant $A \geq 0$ such that for each $k \geq 1$, we have $d_G(g_{N_k},h\gamma_k) \leq A$ almost surely.
\item There exist constants $B,B' \geq 0$ such that for each $T \geq 0$, we have $\E[\sup_{t \in [0,T]} d_G(e,g_t)] \leq BT+B'$.
\item There exist constants $C,C' \geq 0$ such that for each $k \geq 1$, we have $\E[N_k] \leq Ck+C'$.
\eit

We will now describe this construction according to whether $G$ is Archimedean or not.

\ben
\item Consider first the case where $G$ is semisimple real Lie group. Consider the symmetric space $M=G/K$ of $G$, where $K$ is a maximal compact subgroup of $G$, endowed with a $G$-invariant Riemannian metric $d_M$. Without loss of generality, we can assume that the stabilizer of $p_0$ in $\Gamma$ is $\{e\}$. Let $(p_t)_{t  \geq 0}$ denote a standard Brownian motion on $M$, starting from $p_0=[K]$, independent from $h$.  Then $(q_t=h^{-1} \cdot p_t)_{t \geq 0}$ is a standard Brownian motion on $M$, with initial law $\lambda^{-1} \cdot p_0$.

\mk

We will now apply Lyons-Sullivan's discretization procedure to the orbit $\Gamma \cdot p_0$ (see~\cite{lyons_sullivan}, \cite{ballmann_ledrappier}, \cite{kaimanovich}). For a small constant $R>0$, closed balls of radius $2R$ centered at $\Gamma \cdot p_0$ are disjoint. Furthermore, since $M$ has finite volume, the union $F=\bigcup_{\gamma \in \Gamma} \ov{B}(\gamma \cdot p_0,R)$ is recurrent, meaning that the probability that a random path $(q_t)_{t \geq 0}$ intersects $F$ is equal to $1$.

We will follow the description from~\cite{ballmann_ledrappier}. We define an open neighbourhood $V=\bigcup_{\gamma \in \Gamma} \mathring{B}(\gamma \cdot p_0,2R)$ of $F$. Ballmann and Ledrappier define random stopping times $(N_k)_{k \geq 1}$ such that for each $k \geq 1$, $q_{N_k} \in F$ almost surely. The first stopping time $N_1$ is the first entering time to $F$, and for each $k \geq 1$, $N_{k+1} \geq N_k$ is some reentering time to $F$ after having left $V$, but not necessarily the very next one. More precisely, $N_{k+1} \geq N_k$ is the $P^\text{th}$ reentering time to $F$ (after having left $V$), where $P \geq 1$ follows a geometric law with parameter $0<D<1$ (see~\cite{ballmann_ledrappier} for details).

For each $k \geq 1$, let $\gamma_k \in \Gamma$ be the random element such that $d_M(q_{N_k},\gamma_k \cdot p_0) \leq R$ almost surely. The main point of this whole construction is that $\gamma_k$ is the $k^\text{th}$ step of a random walk on $\Gamma$.

Furthermore, since $M$ has finite volume, the expectation of the $n^\text{th}$ reentering time to $F$ is bounded by $B_0n$, where $B_0 \geq 0$ is a constant, so that the expectation of $N_k$ is at most $kB_0D^2$ (see~\cite{ballmann_ledrappier} for details). In particular, there exist constants $C,C' \geq 0$ such that for each $k \geq 1$, we have $\E[N_k] \leq Ck+C'$. Since $M$ has sectional curvature bounded below, the expectation of $d_M(p_0,p_t)$ is at most $B_1t$, where $B_1 \geq 0$ is a constant. Hence for all $k \geq 1$ we have $\E[d_M(p_0,p_{N_k})] \leq kB_0B_1D^2$. 

For each $k \geq 1$, consider a random element $g_k \in G$ such that $g_k \cdot p_0=p_{N_k}$ almost surely. As a consequence, there exist constants $B,B' \geq 0$ such that $\E[d_G(e,g_k)] \leq Bk+B'$. Furthermore, since $d_M(q_{N_k},\gamma_k \cdot p_0) \leq R$ almost surely and $q_{N_k}=h^{-1}g_k \cdot p_0$ almost surely, there exists a constant $A \geq 0$ such that $d_G(g_k,h\gamma_k) \leq A$ almost surely.

\item We will now turn to the case where $G$ is semisimple algebraic group over a non-Archimedean local field. Let $B_G$ denote the Bruhat-Tits building of $G$, and fix a vertex $p_0$ of $B_G$. Since $\Gamma$ is residually finite, we can assume up to replacing $\Gamma$ by a finite index subgroup that the stabilizer of $p_0$ in $\Gamma$ is $\{e\}$.

If $G$ acts transitively on the vertices of its Bruhat-Tits building $B_G$, let $M$ denote the $1$-skeleton of $B_G$. Otherwise, consider the graph $M$ with vertex set $G \cdot p_0$, with an edge in $M$ between two vertices $p \neq p'$ if $p'$ is the closest vertex to $p$, among $G \cdot p_0 \bs \{p\}$, with respect to the combinatorial distance on the $1$-skeleton of $B_G$. Let $d_M$ denote the combinatorial distance on the graph $M$.

For the (left) action of $G$ on $M$ by simplicial automorphisms, the metric $d_M$ is $G$-invariant. Let $(p_t)_{t  \in \N}$ denote the simple random walk on $M$, with uniform probability transitions among all neighbours, starting from $p_0$. Then $(q_t=h^{-1} \cdot p_t)_{t \in \N}$ is a standard random walk on $M$, with initial law $\lambda^{-1} \cdot p_0$.

\mk

We will prove that the induced Markov chain on the (countable) quotient $\Gamma \bs M$ is positively recurrent. First note that if $\Gamma$ is a uniform lattice in $G$, then $\Gamma \bs M$ is a finite connected graph so the result follows. So we now consider the case where $\Gamma$ is a possibly non-uniform lattice in $G$.

Consider a finite set $S \subset G$, such that $s \in S \mapsto s \cdot p_0 \in M$ is a bijection onto the set of neighbours of $p_0$ in $M$. Let $\mu_S$ denote the uniform probability measure on $S$. Let $K$ denote the stabilizer of $p_0$ in $G$: it is a compact subgroup of $G$, let $\lambda_K$ denote its Haar probability measure.

Consider the probability measure $\mu=\lambda_K \mu_S \lambda_K$ on $G$. Then the random walk $(g_t)_{t \in \N}$ on $G$ starting from $e$ with transition law $\mu$ is such that $(g_t \cdot p_0)_{t \in \N}$ is a simple random walk on $M$. Without loss of generality, we can assume that the simple random walk $(p_t)_{t  \in \N}$ is obtained that way, so that $\forall t \in \N, g_t \cdot p_0=p_t$ almost surely.

Note that the Haar probability measure $\lambda$ on $\Gamma \bs G$ is stationary with respect to the right multiplication by $\mu$. As $\lambda$ is invariant under right multiplication by $K$, it defines a probability measure $\ov{\lambda}$ on $\Gamma \bs G / K \simeq \Gamma \bs M$, which is stationary with respect to the simple random walk.

Since $\Gamma \bs M$ is a countable and connected graph, the existence of the stationary probability measure $\ov{\lambda}$ ensures that the random walk $(p_t)_{t  \in \N}$ is recurrent, i.e. if $T=\inf\{t \geq 1 \st p_t=p_0\}$, we have $T < \infty$ almost surely. Furthermore the random walk is positively recurrent, i.e. we have $\E[T] = \f{1}{\ov{\lambda}(\{p_0\})} <\infty$.

\mk

Let $N_1=\inf\{t \geq 1 \st q_t = p_0\}$, and for each $k \geq 1$ let $N_{k+1}=\inf\{t \geq N_k+1 \st q_t=p_0\}$. For each $k \geq 1$, let $\gamma_k \in \Gamma$ denote the unique random element such that $q_{N_k}=\gamma_k \cdot p_0$ almost surely. Then $(\gamma_k)_{k \geq 1}$ is a symmetric random walk on $\Gamma$.

\mk

In particular, since $h^{-1}g_k \cdot p_0=q_{N_k}=\gamma_k \cdot p_0$ almost surely, there exists a constant $A \geq 0$ such that $\forall k \geq 1, d_G(g_k,h\gamma_k) \leq A$ almost surely. Furthermore, $\E[N_1]=C$ and $\E[N_2-N_1]=C'$ are finite since the random walk $(q_t)_{t \in \N}$ is positively recurrent, so for each $k \geq 1$, we have $\E[N_k] \leq Ck+C'$. And there exist constants $B,B' \geq 0$ such that $\E[d_G(e,g_{N_k})] \leq Bk+B'$ for every $k \geq 1$.
\een

\mk

\mk

We will now finish the proof in the general case.

\mk

There exists a finite set $S \subset \Gamma$ such that $B_G(e,A) \subset US$. Let $A'=\max_{s \in S} d_X(s \cdot x_0,x_0)$. Then, for every $k \geq 1$, there exists $s_k \in S$ such that $\chi(g_k^{-1}h)=s_k\gamma_k^{-1}$ almost surely. Hence $d_X(\chi(g_k^{-1}h)^{-1} \cdot x_0,x_0)=d_X(\gamma_k s_k^{-1} \cdot x_0,x_0) \geq d_X(\gamma_k \cdot x_0,x_0)-A'$ almost surely.

\mk

We will now consider the induced action of $G$ on the coarse median space $Y$, as in Section~\ref{sec:induction}. Note that the integrability condition~(\ref{eq:integrability}) is satisfied (using, in case $\Gamma$ is non-uniform, the Margulis arithmeticity theorem (see~\cite{margulis}) and Shalom's work~\cite{shalom}).

\mk

Let us compute, for $k \geq 1$, the expectation $E_k=\E\left[d_Y(g_k \cdot y_0,y_0)\right]$. Notice that since $\E[d_G(e,g_k)] \leq Bk+B'$, and since the action of $G$ on $Y$ is coarsely Lipschitz by Lemma~\ref{lem:orbit coarse lipschitz}, the expectation $E_k$ is finite. Furthermore
\beq E_k&=&\E\left[d_Y(g_k \cdot f_0,f_0)\right] \\
&=&\E\left[\int_U d_X(\chi(g_k^{-1}u)^{-1} \cdot x_0,x_0)d\lambda(u)\right]\\
&=&\E\left[ d_X(\chi(g_k^{-1}h)^{-1}\cdot x_0,x_0) \right]\\
&\geq &\E\left[ d_X(\gamma_k \cdot x_0,x_0) \right]-A'.\eeq

\mk

We will now prove that $(E_k)_{k \geq 1}$ is sublinear in $k$. Since the map $g \in G \mapsto g \cdot f_0 \in Y$ is coarsely Lipschitz by Lemma~\ref{lem:orbit coarse lipschitz}, we can apply Theorem~\ref{thm:sublinear}. As a consequence, we know that there exists a sublinear function $\phi:\R_+ \ra \R_+$ such that $\forall g \in G, d_Y(g \cdot f_0,f_0) \leq \phi(d_G(e,g))$. Up to replacing $\phi$ by its concave hull, we can assume that $\phi$ is concave and non-decreasing. Then we deduce that
$$\forall k \geq 1, E_k=\E[d_Y(g_k \cdot f_0,f_0)] \leq \phi(\E[d_G(e,g_k)]) \leq \phi(Bk+B').$$
In particular, $(E_k)_{k \geq 1}$ is sublinear in $k$.

\mk

In conclusion, since $\E\left[ d_X(\gamma_k \cdot x_0,x_0) \right] \leq A'+E_k$, we deduce that $(\E\left[ d_X(\gamma_k \cdot x_0,x_0) \right])_{k \geq 1}$ is sublinear in $k$. In particular, the random walk ($\gamma_k \cdot x_0)_{k \geq 1}$ on $X$ has zero drift.\ep

\subsection{Random walks on hyperbolic spaces}

We can now finish the proof of Theorem~\ref{thm:main}, using the following result of Maher and Tiozzo:

\bthm[Maher-Tiozzo \cite{maher_tiozzo}] \label{thm:maher tiozzo}
Let $\Gamma$ be a countable group of isometries of a separable Gromov hyperbolic space $X$, let $\nu$ be a non-elementary probability distribution on $\Gamma$, and let $x_0 \in X$ a basepoint. Then a random walk $(\gamma_n)_{n \in \N}$ on $\Gamma$ with step law $\nu$ has positive drift, i.e.
$$\liml_{n \ra \pif} \f{\E\left[d_X(\gamma_n \cdot x_0,x_0)\right]}{n} >0.$$
\ethm

With the notations of Theorem~\ref{thm:main}, assume that $\Gamma$ acts by isometries of a Gromov-hyperbolic space $X$. Up to passing to the injective hull of of $X$ (see~\cite{lang}), we can assume that $X$ is geodesic. Up to passing to a convex subset of $X$ containing some orbit of $\Gamma$, we can assume that $X$ is also separable. 

Then according to Proposition~\ref{pro:zero drift}, there exists a probability measure $\nu$ on $\Gamma$ with support generating $\Gamma$, such that the associated random walk on $X$ has zero drift. According to Theorem~\ref{thm:maher tiozzo}, this implies that the action of $\Gamma$ on $X$ is elementary.

If the action of $\Gamma$ on $X$ was lineal, then it would give an unbounded quasimorphism from $\Gamma$ to $\R$. According to Burger and Monod (see~\cite{burger_monod}), any quasi-morphism from $\Gamma$ to $\R$ bounded.

As a consequence, the action of $\Gamma$ on $X$ is elliptic or parabolic. This concludes the proof.

\section{Proof of corollaries}

\label{sec:cor}

We start by recalling the definition of a quasi-action, as in~\cite{manning}.

\bdf[Quasi action]
Let $(X,d)$ be a metric space, and let $(G,d_G)$ be a group endowed with a left invariant distance. A map $G \times X \ra X : (g,x) \mapsto g \cdot x$ is called a \emph{quasi-action} if there exist constants $(K,C)$ such that the following hold
\ben
\item For each $g \in G$, the map $X \ra X : x \mapsto g \cdot x$ is a $(K,C)$-quasi-isometry.
\item For each $x \in G$, the map $G \ra X : g \mapsto g \cdot x$ is coarsely $(K,C)$-Lipschitz.
\item For each $x \in X$ and $g,h \in G$, we have $d(g\cdot (h \cdot x),(gh)\cdot x) \leq C$.
\een
\edf

We will recall the following.

\bpro[Manning \cite{manning}] \label{pro:quasi tree}
Assume that a finitely generated group $\Gamma$ has a quasi-action on tree $X$. There exists a quasi-tree $X'$ such that $\Gamma$ acts by isometries on $X'$. Furthermore, $X'$ is quasi-equivariantly quasi-isometrically embedded into $X$.
\epro

We can now give the proof of Corollary~\ref{cor:qfa}.

\bp[Proof of Corollary~\ref{cor:qfa}]
Assume that $\Gamma$ has a quasi-action on a tree $X$. According to Proposition~\ref{pro:quasi tree}, $\Gamma$ has an action on a quasi-tree $X'$. According to Theorem~\ref{thm:main}, this action is elliptic or parabolic. Since $\Gamma$ is finitely generated, it has no parabolic action on a quasi-tree. As a consequence, the action of $\Gamma$ on $X'$ is elliptic, so the quasi-action of $\Gamma$ on $X$ has bounded orbits.\ep

We can now give the proof of Corollary~\ref{cor:mcg}. As explained in the introduction, we will only use that every action of a higher rank lattice on a hyperbolic space is elementary, and that higher rank lattices do not surject onto $\Z$ (which is a direct consequence of Property (T)).

\bp[Proof of Corollary~\ref{cor:mcg}]
Consider a morphism $\phi : \Gamma \ra MCG(S)$, where $S$ is a closed surface of genus $g$, with $p$ punctures. We can assume that $MCG(S)$ is infinite. Let $H=\phi(\Gamma)$.

The curve graph $\cal{C}(S)$ is hyperbolic by~\cite{masur_minsky}, so by Theorem~\ref{thm:main}, the action of $H$ on $\cal{C}(S)$ is elementary. 

According to~\cite{ivanov}, any subgroup of $MCG(S)$ having an elementary action on $\cal{C}(S)$ is either virtually cyclic or reducible. Since no finite index subgroup of $\Gamma$ surjects onto $\Z$, $H$ is not virtually cyclic. As a consequence, $H$ is reducible: some finite index subgroup $H_0$ fixes a curve $c$. Observe that the stabilizer of $c$ in $MCG(S)$ is a (product of) mapping class groups of surfaces of smaller complexities. By induction, one sees that $H$ is in fact finite.\ep

We can now give the proof of Corollary~\ref{cor:hhg}, which is exactly the same proof as the previous one, written in the more technical context of hierarchically hyperbolic groups.

\bp[Proof of Corollary~\ref{cor:hhg}]
Consider a morphism $\phi : \Gamma \ra G$, where $G$ is a hierarchically hyperbolic group. Let $\frak{S}$ denote the index set of $G$, let $S \in \frak{S}$ denote the maximally nested element, and let $\cal{C}(S)$ denote its associated hyperbolic space.

Let $H=\phi(G)$. The group $H$ acts by isometries on the hyperbolic space $\cal{C}(S)$: according to Theorem~\ref{thm:main}, the action of $H$ on $\cal{C}(S)$ is elementary. According to~\cite[Corollary~14.4]{hhg1}, the action of $G$ on $\cal{C}(S)$ is acylindric. So if $H$ has unbounded orbits, then $H$ is virtually cyclic by~\cite[Theorem~1.1]{osin}.

Since no finite index subgroup of $\Gamma$ surjects onto $\Z$, $H$ is not virtually cyclic. As a consequence, $H$ has bounded orbits in $\cal{C}(S)$.

According to the proof of~\cite[Theorem~9.15]{hhg4}, there exists $U \in \frak{S}$, $U \subsetneq S$, such that some finite index subgroup $H_0$ of $H$ fixes $U$. By induction on complexity, we conclude that $H$ is in fact finite.\ep

We finish with the proof of Corollary~\ref{cor:ah}.

\bp[Proof of Corollary~\ref{cor:ah}]
Consider a morphism $\phi : \Gamma \ra G$, where $G$ is an acylindrically hyperbolic group. Consider an acylindrical action of $G$ on a hyperbolic space $X$. Then according to Theorem~\ref{thm:main}, the action of $\phi(\Gamma)$ on $X$ is elliptic or parabolic. According to~\cite{osin}, there are no acylindrical parabolic actions on a hyperbolic space. As a consequence, the action of $\phi(\Gamma)$ on $X$ is elliptic.\ep

\sign

\newpage

\begin{center}
{\Large Appendix: Morphisms from higher rank lattices to $\text{Out}(F_N)$} \\
\bigskip
{\large Vincent Guirardel and Camille Horbez}
\end{center}

\bigskip
\bigskip

In this appendix,  we use Theorem \ref{thm:main} to show that homomorphisms from higher rank lattices $\Gamma$
to $\Out(G)$ have finite image when $G$ is a free group, a torsion-free hyperbolic group, 
and even a relatively hyperbolic group or a right-angled Artin group under suitable additional assumptions.
This was first proved by Bridson--Wade \cite{BW} for $\Out(F_N)$ and by Wade \cite{Wad} for right-angled Artin groups, for a more general class of groups $\Gamma$.
Note that their approach is based on the algebraic structure of the Torelli group, which is not available for hyperbolic groups.
A crucial step in what we do consists in understanding the case where $G$ is a free product.

\paragraph*{Statement of the main result.} Let $G$ be a countable group that splits as a free product of the form $$G=G_1\ast\dots\ast G_k\ast F_N,$$ where $F_N$ denotes a free group of rank $N$. We denote by $\text{Out}(G,\{G_i\})$ the subgroup of $\text{Out}(G)$ made of those automorphisms that preserve (setwise) the conjugacy classes of the subgroups $G_i$, and by $\text{Out}(G,\{G_i\}^{(t)})$ the subgroup made of automorphisms 
whose restriction to each $G_i$ coincides with the conjugation by an element $g_i\in G$.
Given a group $H$, we denote by $Z(H)$ the center of $H$.

\begin{theo}\label{free-product-1}
Let $\Gamma$ be a lattice in a product of higher rank almost simple connected algebraic groups with finite center over local fields. Let $G$ be a countable group that splits as a free product of the form $$G=G_1\ast\dots\ast G_k\ast F_N.$$
\\ Assume that for all $i\in\{1,\dots,k\}$, and every finite index subgroup $\Gamma_0\subseteq\Gamma$, every homomorphism from $\Gamma_0$ to $G_i/Z(G_i)$ has finite image.
\\ Then every homomorphism from $\Gamma$ to $\text{Out}(G,\{G_i\}^{(t)})$ has finite image.
\end{theo}

Before we prove Theorem~\ref{free-product-1}, we start by mentioning its consequences. 

\paragraph*{Automorphisms of free groups.} First, we notice that in the particular case where there is no peripheral group $G_i$, we obtain the following result due to Bridson--Wade.

\begin{coro}[Bridson--Wade \cite{BW}]\label{BW}
Let $\Gamma$ be a lattice in a product of higher rank almost simple connected algebraic groups with finite center over local fields.
\\ Then every homomorphism from $\Gamma$ to $\text{Out}(F_N)$ has finite image.
\qed
\end{coro}

\paragraph*{Automorphisms of (relatively) hyperbolic groups.}

\begin{coro}\label{out-hyp}
Let $G$ be a torsion-free group which is hyperbolic relative to a finite collection of  finitely generated subgroups $P_1,\dots,P_k$. Let $\Gamma$ be a lattice in a product of higher rank almost simple connected algebraic groups with finite center over local fields. 
\\ Assume that for all $i\in\{1,\dots,k\}$, and for any finite index subgroup $\Gamma_0$ of $\Gamma$,
\begin{enumerate}
\item every homomorphism from $\Gamma_0$ to $P_i/Z(P_i)$ has finite image,
\item every homomorphism from $\Gamma_0$ to $\text{Out}(P_i)$ has finite image.
\end{enumerate}
\noindent Then every homomorphism from $\Gamma$ to $\text{Out}(G,\{P_i\})$ has finite image.
\end{coro}

A particular case of Corollary \ref{out-hyp} is the following result, stated in the introduction.

\setcounter{mainthm}{5}
\begin{maincor}
Let $G$ be a torsion-free Gromov hyperbolic group. Let $\Gamma$ be a lattice in a product of higher rank almost simple connected algebraic groups with finite center over local fields.
\\ Then every homomorphism from $\Gamma$ to $\text{Out}(G)$ has finite image.
\qed
\end{maincor}

We will use the following simple result several times.

\begin{lemma}\label{lem_centre}
Let $P$ be a group, let $\phi:\Gamma\to P$ be a morphism, and let $N\normal P$ an abelian normal subgroup such that the image of $\Gamma$ in $P/N$ is finite.
\\ Then $\phi(\Gamma)$ is finite.
\end{lemma}

\begin{proof}
  The hypothesis implies that $\Gamma$ has a finite index subgroup $\Gamma_0$ such that $\phi(\Gamma_0)\subset N$.
Since $\Gamma_0$ has finite abelianization, $\phi(\Gamma)$ is finite.
\end{proof}

\begin{proof}[Proof of Corollary~\ref{out-hyp}] Let $\calp=\{P_1,\dots,P_k\}$.  Let $\rho:\Gamma\to\text{Out}(G,\calp)$ be a morphism.
\paragraph{Case 1:} $G$ is freely indecomposable relative to the parabolic subgroups, i.e. $G$ has no decomposition into a free product
in which each $P_i$ is conjugate into a factor. 
\\ Let $\Lambda$ be the canonical elementary JSJ decomposition of $G$ relative to $\calp$ (see \cite[Theorem 4]{GL4}, \cite{Bow_JSJ} when $G$ is hyperbolic).
In this case $\text{Out}(G)$  has a finite index subgroup $\text{Out}^1(G)$ which is an extension of a finite product of mapping class groups of compact surfaces and subgroups of the outer automorphism groups $\text{Out}(P_i)$ by the group $\calt$ of twists of $\Lambda$ \cite[Theorem~4.3]{GL1}. Let $\Gamma_0$ be the finite index subgroup of $\Gamma$ made of all elements whose $\rho$-image lies in $\text{Out}^1(G)$. 
Using Farb--Kaimanovich--Masur's theorem (Corollary~\ref{cor:mcg}), together with our second hypothesis stating that every morphism from $\Gamma_0$ to $\text{Out}(P_i)$ has finite image, we get that the image of some finite index subgroup $\Gamma_1$ of $\Gamma$ is contained in $\calt$.
When $G$ is a torsion-free hyperbolic group, $\calt$ is an abelian group, which concludes the proof in this case.
In general, Lemma \ref{lem_twist} below shows that  there is a morphism from $\calt$ to a product of copies of $P_i/Z(P_i)$, whose kernel is abelian.
By hypothesis, any morphism from $\Gamma_1$ to $P_i/Z(P_i)$ has finite image. Applying Lemma \ref{lem_centre}, we get that $\Gamma_1$
has finite image in $\calt$.

\paragraph{General case.}
Consider a Grushko decomposition $$G=G_1\ast\dots\ast G_k\ast F_N$$ of $G$ relative to the parabolic subgroups: this is a decomposition of $G$ as a free product in which all subgroups in $\mathcal{P}$ are conjugate into one of the factors, where each $G_i$ is nontrivial, freely indecomposable relative to $\mathcal{P}_{|G_i}$, 
and not isomorphic to $\mathbb{Z}$, 
(here $\mathcal{P}_{|G_i}$ is defined as a choice of a conjugate of each $P_j$ contained in $G_i$ if it exists). Every subgroup $G_i$ is hyperbolic relative to $\mathcal{P}_{|G_i}$. 

Let $\Gamma_0<\Gamma$ be the finite index subgroup of elements whose $\rho$-image lies in the group 
 $\text{Out}^0(G,\calp)$ made of automorphisms that preserve 
 the conjugacy class of each subgroup $G_i$.
Since all subgroups $G_i$ are their own normalizers, there is a morphism $\text{Out}^0(G,\calp)\to\prod_{i=1}^k\text{Out}(G_i,\mathcal{P}_{|G_i})$ whose kernel is
$\text{Out}(G,\{G_i\}^{(t)})$. By Case 1, the image of $\Gamma_0$ in $\prod_{i=1}^k\text{Out}(G_i,\mathcal{P}_{|G_i})$ is finite
so there exists a finite index subgroup $\Gamma_1<\Gamma$ whose image in $\Out(G,\calp)$ is contained in $\text{Out}(G,\{G_i\}^{(t)})$.

To apply Theorem \ref{free-product-1}, 
let us check that for every finite index subgroup $\Gamma_2\subseteq\Gamma_1$, every homomorphism $\phi:\Gamma_2\to G_i/Z(G_i)$ has finite image.
If $G_i$ is elementary, then it is either cyclic or equal to a conjugate of some $P_i$, so this holds by assumption.
Otherwise, $Z(G_i)$ is trivial and by Theorem \ref{thm:main}, the image of $\Gamma_2$ is finite or parabolic.
In view of Lemma \ref{lem_centre}, our assumption implies that  $\phi(\Gamma_2)$ is finite.
Thus, Theorem \ref{free-product-1} applies  and concludes the proof.
\end{proof}

\begin{lemma} \label{lem_twist}
Let $G$ be a torsion-free group with is hyperbolic relative to $\calp=\{P_1,\dots, P_k\}$, and freely indecomposable relative to $\calp$. 
Let $\calt$ be the group of twists of the canonical elementary JSJ decomposition of $G$ relative to $\calp$.
Then $\calt$ maps with abelian kernel to a direct product of copies of $P_i/Z(P_i)$. 
\end{lemma}
 
\begin{proof}
 We consider $\Lambda$ the canonical elementary JSJ decomposition of $G$ relative to $\calp$ as described in \cite[Theorem 4]{GL4}.
We  follow \cite[\S 3]{Lev} for the following description of $\calt$. We denote by $V,\vec E,E$ the set of vertices, oriented edges and non-oriented edges  of $\Lambda$.
Then $\calt$ is isomorphic to the quotient $\Tilde \calt/N$ where 
$$\Tilde \calt=\prod_{e\in \vec E} Z_{G_{t(e)}}(G_e)$$ 
and $N=\grp{N_V,N_E}\normal\Tilde \calt$ is a central subgroup generated by $N_V,N_E$ defined as follows. 
The group $N_V=\prod_{v\in V} Z(G_v)$ is embedded in $\Tilde \calt$
by sending $Z(G_v)$ diagonally in $\prod_{t(e)=v} Z_{G_{t(e)}}(G_e)\subset \Tilde \calt$,
and $N_E=\prod_{\bar e\in E} Z(G_{\bar e})$  is embedded in $\Tilde \calt$ by sending $Z(G_{\bar e})$ diagonally in 
$Z_{G_{t(\overrightarrow e)}}(G_{\overrightarrow e})\times Z_{G_{t(\overleftarrow e})}(G_{\overleftarrow e})\subset\Tilde \calt$,
where $\overrightarrow e,\overleftarrow e$ are the two orientations of the non-oriented edge $\bar e\in E$. 

Now the canonical JSJ decomposition of $G$ is bipartite, where each edge joins a vertex with nonelementary stabilizer to a vertex
which is maximal elementary (i.e. maximal loxodromic, in particular cyclic, or conjugate to some $P_i$).
Denote by $V=V_{ne}\dunion V_{el}$ the corresponding partition of the vertices.
It follows that for each $e\in E$ such that $t(e)\in V_{ne}$, we have $Z_{G_{t(e)}}(G_e)=Z(G_e)$ (indeed, $\grp{G_e,Z_{G_v}(G_e)}$ is elementary and is therefore contained in the maximal elementary subgroup $G_{o(e)}$,
so $\grp{G_e,Z_{G_v}(G_e)}\subset G_e$). Thus, 
$$\Tilde \calt/N_E\simeq \prod_{e\in \vec E_{el}} Z_{G_{t(e)}}(G_e)$$ 
where  $\vec E_{el}\subset\vec E$ is the set of edges $e$ such that $t(e)\in V_{el}$.

Since $Z(G_v)$ is trivial for each $v\in V_{ne}$, the group $\calt$ is isomorphic to the
quotient of $\prod_{e\in \vec E_{el}} Z_{G_{t(e)}}(G_e)$ by the diagonal embedding of $\prod_{v\in V_{el}} Z(G_v)$. 
Moding out by the larger central subgroup $\prod_{e\in \vec E_{el}} Z(G_{t(e)})$,
we get that $\calt$ maps with central kernel to 
$$\prod_{e\in \vec E_{el}} Z_{G_{t(e)}}(G_e)/Z(G_{t(e)})\subset \prod_{e\in \vec E_{el}} G_{t(e)}/Z(G_{t(e)})=\prod_{v\in V_{el}} (G_v/Z(G_v))^{d_v}$$
where $d_v$ is the degree of the vertex $v$.
Now for $v\in V_{el}$, the group $G_v$ is either cyclic (in which case $G_v/Z(G_v)$ is trivial), or conjugate to a parabolic group $P_i$. This proves the lemma.
\end{proof}

\paragraph*{Automorphisms of right-angled Artin groups.}
Theorem~\ref{free-product-1} also enables us to find a new proof of Wade's theorem  about morphisms with values in the automorphism group of a right-angled Artin group \cite{Wad}. 
Given a finite simplicial graph $X$, the right-angled Artin group $A_X$ is defined as the group with one generator for each vertex in $X$,
and a commutation relation between each pair of vertices joined by an edge.

The \emph{SL-dimension} $d_{SL}(A_X)$ is defined as the maximal size of a clique in $X$ made of vertices that all have the same star in $X$. 
Note that $\Out(A_X)$ contains a group isomorphic to $GL(d,\bbZ)$ for $d=d_{SL}(A_X)$.

\begin{coro}[Wade \cite{Wad}]\label{Wade}
Let $\Gamma$ be a lattice in a product of higher rank almost simple connected algebraic groups with finite center over local fields.
\\ Let $d\in\mathbb{N}$ be such that  every homomorphism from a finite index subgroup $\Gamma_0<\Gamma$ to $GL(d,\mathbb{Z})$ has finite image.\\
Then for any  right-angled Artin group $A$ with $d_{SL}(A)\leq d$, 
any homomorphism from $\Gamma$ to $\text{Out}(A)$ has finite image.
\end{coro}


\begin{proof}
Given a graph $X$, we say that a partial order $\prec$ on the vertex set of $X$ is \emph{admissible} if we have $\text{lk}(v)\subseteq\text{st}(w)$ whenever $v\prec w$. 
We define $\text{Out}^0(A_X,\prec)$ to be the subgroup of $\text{Out}(A_X)$ generated by partial conjugations and transvections of the form $v\mapsto vw$ with $v\prec w$. In particular, if $\prec_{\max}$ is the order defined by declaring that $v\prec_{\max}w$ whenever $\text{lk}(v)\subseteq\text{st}(w)$, then it follows from \cite{Laurence} that $\text{Out}^0(A_X,\prec_{\max})$ is a normal subgroup of finite index in $\text{Out}(A_X)$.
 We define $d_{SL}(X,\prec)$ as the maximal size of a clique in $X$ made of vertices that are pairwise $\prec$-equivalent (two vertices $v,w$ are $\prec$-equivalent if $v\prec w$ and $w\prec v$). In particular $d_{SL}(A_X)=d_{SL}(X,\prec_{\max})$. We note that if $Y\subset X$ is an induced subgraph (i.e. whenever $Y$ contains two vertices of $X$ joined by an edge in $X$, then $Y$ also contains this edge), then the restriction $\prec_{|Y}$ of $\prec$ to $Y$ is an admissible partial order on $Y$, and $d_{SL}(Y,\prec_{|Y})\leq d_{SL}(X,\prec)$.

Fix $d\geq 0$ and $\Gamma$ a lattice as in the statement. 
We will prove by induction on the number of vertices in $X$ that  if $(X,\prec)$ is a graph with an admissible partial ordering such that $d_{SL}(X,\prec)\leq d$, then every morphism $\rho:\Gamma\to\text{Out}^0(A_X,\prec)$ has finite image. 
Since $\text{Out}^0(A_X,\prec_{\max})$ has finite index in $\Out(A_X)$, the result will follow.

To prove the claim, first assume that $X$ is disconnected. The Grushko decomposition of $A_X$ is of the form $$A_X=A_{X_1}\ast\dots\ast A_{X_k}\ast F_N,$$ where $X_1,\dots,X_k$ are the connected components of $X$ which are not reduced to a point, and $X$ has $N$ connected components reduced to a point. Any automorphism in  $\text{Out}^0(A_X,\prec)$ preserves the conjugacy class of each $A_{X_i}$. Since $A_{X_i}$ is its own normalizer, 
there is a restriction map $r:\text{Out}^0(A_X,\prec)\to \prod_{i=1}^k\text{Out}(A_{X_i})$ whose kernel is contained in $\text{Out}(A_X,\{A_{X_i}\}^{(t)})$.

Looking at the image of the generators, we see that the image of $r$ is contained in $\prod_{i=1}^k\text{Out}^0(A_{X_i},\prec_{|X_i})$.
Since $d_{SL}(A_{X_i},\prec_{|X_i})\leq d_{SL}(X,\prec)\leq d$, our induction hypothesis shows that $r\circ \rho(\Gamma)$ is finite.
Thus, for some finite index subgroup $\Gamma_0\subset \Gamma$, we have $\rho(\Gamma_0)\subset \text{Out}(A_X,\{A_{X_i}\}^{(t)})$. To deduce that $\rho$ has finite image, it suffices to check that we can apply Theorem~\ref{free-product-1}.
Since a right-angled Artin group is the direct product of its center by another right-angled Artin group,
it is enough to check that any morphism from $\Gamma_0$ to a right-angled Artin group is trivial.
This follows from the fact that $\Gamma_0$ has property $(T)$, 
and that right-angled Artin groups are cubical.

We now assume that $X$ is connected. First, if the center of $A_X$ is nontrivial, then it is generated by the vertices in a clique $C\subset X$. 
By \cite[Proposition 4.4]{CV}, there is a morphism
 $$\Psi:\text{Out}^0(A_X,\prec_{\max})\to \Out(A_{C})\times\text{Out}(A_{X\setminus C})$$ whose kernel is free abelian.
By Lemma \ref{lem_centre}, it is enough to check that the image of $\Gamma$ in $\Out(A_{C})\times\text{Out}(A_{X\setminus C})$ is finite. 
 By looking at the images of the generators, we see that  $$\Psi(\text{Out}^0(A_X,\prec))\subset\Out^0(A_{C},\prec_{|C})\times\text{Out}^0(A_{X\setminus C},\prec_{|X\setminus C}).$$
 Notice that $\Out^0(A_{C},\prec_{|C})$ is isomorphic to a block-triangular subgroup of $SL(\#C,\mathbb Z)$, and the maximal size of a block is $d(C,\prec_{|C})\leq d$. 
Therefore the image of $\Gamma$ in $SL(\#C,\mathbb{Z})$ is virtually unipotent, hence finite
since finite index subgroups of $\Gamma$ have finite abelianization.
The fact that any morphism from $\Gamma$ to $\text{Out}^0(A_{X\setminus C},\prec_{|X\setminus C})$ follows from our induction hypothesis and we are done in this case.

We finally assume that $Z(A_X)$ is trivial. By  \cite[Corollary 3.3]{CV}, there is a morphism 
$$\Psi:\text{Out}^0(A_X)\to \prod\text{Out}^0(A_{\text{lk}([v])})$$   where the product is taken over all maximal equivalence classes of vertices $[v]$ for the order $\prec_{\max}$. The kernel $K$ of $\Psi$ is a free abelian group \cite[Theorem 4.2]{CV}.
By looking at the image of the generators, we see that
$$\Psi(\text{Out}^0(A_X,\prec) )\subset \prod\text{Out}^0(A_{\text{lk}([v])},\prec_{|\text{lk}([v])}).$$
By induction, the image of $\rho(\Gamma)$ under $\Psi$ is finite. 
Lemma \ref{lem_centre} concludes the proof.
\end{proof}

\paragraph*{Background on free products and their automorphisms.} The rest of this appendix is devoted to the proof of Theorem~\ref{free-product-1}. We start with some background on free products and their automorphism groups. We denote by $\calf$ the collection of all conjugacy classes of the subgroups $G_i$, and write $\text{Out}(G,\calf)$ and $\text{Out}(G,\calf^{(t)})$ instead of $\text{Out}(G,\{G_i\})$ and $\text{Out}(G,\{G_i\}^{(t)})$. A subgroup of $G$ is \emph{peripheral} if it is conjugate into one of the subgroups $G_i$.

A theorem of Kurosh \cite{Kur} states that every subgroup $H\subseteq G$ inherits a free product decomposition $H=(\ast_{j\in J} H_j)\ast F$, where each $H_j$ is conjugate to a subgroup of one of the peripheral subgroups $G_i$, and $F$ is a free subgroup of $G$. We denote by $\calf_{|H}$ the collection of all $H$-conjugacy classes of the subgroups $H_j$.

A \emph{$(G,\calf)$-tree} is an $\mathbb{R}$-tree $T$ equipped with a $G$-action, such that every peripheral group $G_i$ fixes a point in $T$. A \emph{$(G,\calf)$-free splitting} is a minimal (i.e. without proper invariant subtree), simplicial $(G,\calf)$-tree with trivial edge stabilizers. A \emph{$(G,\calf)$-free factor} is a subgroup of $G$ that coincides with a point stabilizer in some $(G,\calf)$-free splitting. More generally, a \emph{free factor system} of $(G,\calf)$ is a collection of subgroups of $G$ that arises as the collection of all nontrivial point stabilizers in a $(G,\calf)$-free splitting. 
A free factor system $\calf$ is \emph{smaller} than $\calf'$ if any group in $\calf$ is conjugate into a group in $\calf'$ (equivalently,
the free splitting defining  $\calf$ dominates the one defining $\calf'$).
A $(G,\calf)$-free factor is \emph{proper} if it is nonperipheral (in particular nontrivial) and not equal to $G$. In the Kurosh decomposition inherited by a free factor $A$, the set $J$ is finite, and the free group $F$ is finitely generated.

A \emph{relative $\calz$-splitting} is a minimal, simplicial $(G,\calf)$-tree with edge stabilizers trivial or cyclic and nonperipheral. The \emph{graph of relative $\calz$-splittings}, denoted by $FZ(G,\calf)$, is the graph (equipped with the simplicial metric) whose vertices are the homeomorphism classes of relative $\calz$-splittings, with an edge between two splittings $S,S'$ if they have a common refinement (i.e. there exists a relative $\calz$-splitting $\hat{S}$ which admits $G$-equivariant alignment-preserving maps onto both $S$ and $S'$). Hyperbolicity of the graph of relative $\calz$-splittings was first proved by Mann in the context of free groups \cite{Man}, and extended to the general case in \cite{Hor3}. The group $\text{Out}(G,\calf)$ has a natural action on $FZ(G,\calf)$.

\paragraph*{Proof of the main theorem.} We start by  stating two lemmas that will be useful in our proof of Theorem~\ref{free-product-1}.

Since $G_i^d/Z(G_i)$ maps to $(G_i/Z(G_i))^d$ with central kernel, Lemma \ref{lem_centre} yields the following statement.
\begin{lemma}\label{lem_1}
Under the hypotheses of Theorem~\ref{free-product-1}, for every finite index subgroup $\Gamma_0\subseteq\Gamma$,
\begin{enumerate}
\item Every morphism from  $\Gamma_0$ to $G_i$ has finite image.
\item For all $d\in\mathbb{N}$, every morphism from $\Gamma_0$ to $G_i^d/Z(G_i)$, where $Z(G_i)$ sits in $G_i^d$ via the diagonal inclusion map, has finite image.  \qed
\end{enumerate}
\end{lemma}


A $(G,\calf)$-tree is \emph{very small} if pointwise stabilizers of nondegenerate arcs in $T$ are either trivial, or cyclic and nonperipheral, and tripod stabilizers are trivial. Our second lemma concerns morphisms from a higher rank lattice to a subgroup of $G$ that stabilizes a point in a very small $(G,\calf)$-tree. 

\begin{lemma}\label{point-stab}
Let $T$ be a very small $(G,\calf)$-tree, and let $G_v\subseteq G$ be a point stabilizer in $T$. Let $d\in\mathbb{N}$. Then under the assumptions of Theorem~\ref{free-product-1}, every morphism from $\Gamma$ to $G_v^d/Z(G_v)$ (where $Z(G_v)$ sits in $G_v^d$ via the diagonal inclusion map) has finite image. 
\end{lemma}

\begin{proof}
As above, it is enough to prove that every morphism $\rho:\Gamma\to G_v/Z(G_v)$ has finite image.
The subgroup $G_v\subseteq G$ inherits a free product decomposition $G_v=(\ast_j H_j)\ast F$, where each $H_{j}$ is conjugate into some peripheral group $G_i$, and $F$ is a free group. 

Since $G_v$ is a point stabilizer in a very small $(G,\calf)$-tree, each subgroup $H_j$ in this decomposition is actually equal to a conjugate of some $G_i$ (it cannot be a proper subgroup of $G_i$): this is because in a very small $(G,\calf)$-tree, every peripheral subgroup
 fixes a unique point. 

The conclusion obviously holds if $G_v$ is isomorphic to $\mathbb{Z}$, and it holds by hypothesis if $G_v$ is a conjugate of one of the subgroups $G_i$. In all other cases, the center $Z(G_v)$ is trivial. It then follows from Theorem~\ref{thm:main} that the image of $\rho:\Gamma\to G_v$ is contained in one of the factors $H_j$, and the first assertion of Lemma~\ref{lem_1} implies that this image is finite. 
\end{proof}

\begin{proof}[Proof of Theorem~\ref{free-product-1}]
We assume that all subgroups $G_i$ are nontrivial. We define the \emph{complexity} of $(G,\calf)$ as 
$\xi(G,\calf):=\max(k-1,0)+N$ (this is the number of edges of any reduced Grushko $(G,\calf)$-tree).
The proof goes by induction on $\xi(G,\calf)$.
Let $\rho:\Gamma\to\text{Out}(G,\calf^{(t)})$ be a homomorphism.
\\
\\
\underline{\textbf{Initialization}} 
\\ We first treat the cases where $\xi(G,\calf)\leq 1$.
\\ $\ast$ The statement is obvious if either $k=1$ and $N=0$ (i.e. $G=G_1$) or $k=0$ and $N=1$ (i.e. $G=\mathbb{Z}$).
\\ $\ast$ If $k=2$ and $N=0$, i.e. $G=G_1\ast G_2$, then by \cite{Lev}, the group $\text{Out}(G,\{G_1,G_2\}^{(t)})$ is isomorphic to $G_1/Z(G_1)\times G_2/Z(G_2)$, and the result follows from our hypothesis that every homomorphism from $\Gamma$ to either $G_1/Z(G_1)$ or $G_2/Z(G_2)$ has finite image.
\\ $\ast$ If $k=1$ and $N=1$, i.e. $G=G_1\ast\mathbb{Z}$, then by \cite{Lev}, the group $\text{Out}(G,\{G_1\}^{(t)})$ has a subgroup of index $2$ isomorphic to $(G_1\times G_1)/Z(G_1)$ (where $Z(G_1)$ sits as a subgroup of $G_1\times G_1$ via the diagonal inclusion map). The result then follows from the second assertion of Lemma~\ref{lem_1}.
\\
\\
\underline{\textbf{Inductive step}}
\\ We now assume that $\xi(G,\calf)\ge 2$. 
 Theorem~\ref{thm:main} ensures that (up to replacing $\Gamma$ by a finite index subgroup) the image $\rho(\Gamma)$ in $\text{Out}(G,\calf^{(t)})$ acts elementarily on the $\calz$-splitting graph $FZ(G,\calf)$. By \cite[Theorem 4.3]{Hor}, either $\rho(\Gamma)$ virtually fixes the conjugacy class of a proper $(G,\calf)$-free factor, or else it virtually fixes the homothety class of a very small $(G,\calf)$-tree with trivial arc stabilizers.

Up to replacing $\Gamma$ by a finite index subgroup, we first assume that $\rho(\Gamma)$ fixes the conjugacy class of a proper $(G,\calf)$-free factor $A$. 
 We denote by $\calf'$ the smallest free factor system of $(G,\calf)$ such that $A\in \calf'$. 
There is a morphism $$\phi:\rho(\Gamma)\to\text{Out}(A,\mathcal{F}_{|A}^{(t)})$$ whose kernel is contained in 
$\text{Out}(G,\calf'^{(t)})$. 
We have $\xi(A,\calf_{|A})<\xi(G,\calf)$ 
so a first application of the induction hypothesis shows that $\phi$ has finite image. Hence $\rho(\Gamma)$ is virtually a subgroup of $\text{Out}(G,\calf'^{(t)})$. We also have $\xi(G,\calf')<\xi(G,\calf)$ 
so a second application of the induction hypothesis shows that $\rho(\Gamma)$ is finite. 

We now assume that $\rho(\Gamma)$ fixes the homothety class of very small $(G,\calf)$-tree $T$ with trivial arc stabilizers. There is a morphism $\lambda:\rho(\Gamma)\to\mathbb{R}_+^\ast$, given by the homothety factor (i.e. $\lambda(\Phi)$ is the unique real number such that $T.\Phi=\lambda(\Phi).T$). The morphism $\lambda$ has finite (hence trivial) image because $\mathbb{R}_+^\ast$ is abelian. Therefore $\rho(\Gamma)$ is contained in the stabilizer $\text{Stab}(T)$ of the isometry class of $T$. Since point stabilizers in $T$ are malnormal in $G$, there is a morphism $\psi$ from $\rho(\Gamma)$ to the direct product of all subgroups $\text{Out}(G_v,\calf_{|G_v}^{(t)})$, where $G_v$ varies among a finite set of representatives of the conjugacy classes of all nontrivial point stabilizers of $T$. The kernel  of $\psi$ is contained in the subgroup $\text{Stab}(T,\{G_v\}^{(t)})$ made of automorphisms that fix the isometry class of $T$ and act by conjugation on each subgroup $G_v$. We also know that $G_v$ is finitely generated \cite[Corollary 4.5]{Hor2}, so the Kurosh decomposition of $G_v$ is finite, i.e. it has finitely many factors, and the free subgroup arising in the decomposition is finitely generated. For all branch points $v$ of $T$, we have $\xi(G_v,\calf_{|G_v})<\xi(G,\calf)$ (see the proof of \cite[Theorem~6.3]{Hor}), so by induction $\psi$ has finite image. So $\rho(\Gamma)$ is virtually a subgroup of $\text{Stab}(T,\{G_v\}^{(t)})$. 
By \cite{GL}, $\text{Stab}(T,\{G_v\}^{(t)})$ virtually injects into a direct product of $G_v^{d_v}/Z(G_v)$, where $d_v$ denotes the degree of $v$ in $T$. It then follows from Lemma~\ref{point-stab} that $\rho$ has finite image, as required.
\end{proof}

\bibliographystyle{smfalpha_perso}
\bibliography{bibli}

\end{document}